\documentclass[10pt,a4paper]{amsart}     
\usepackage{graphicx}
\usepackage{amsmath}
\usepackage{epstopdf} % needed if you have eps figures in a pdflatex manuscript
\usepackage{mathptmx}      % use Times fonts if available on your TeX system

\numberwithin{equation}{section}
\usepackage{amssymb}
\usepackage{amsthm}
\usepackage{enumitem}  %fancy enumerations
\usepackage{etaremune} %reversed numbering in enumerations
\usepackage{url}		%\url{?}
\usepackage{mdframed}
\usepackage{framed}	%\begin{framed}..\end{framed}
\usepackage{hyperref}	%links within document
\usepackage{bbold}		%to get indicator function \mathbb{1}
\usepackage{color}		%\color{} command
\usepackage{graphicx}		

\theoremstyle{definition}
\newtheorem{definition}{Definition}[section]

\theoremstyle{plain}
\newtheorem{theorem}[definition]{Theorem}
\newtheorem{lemma}[definition]{Lemma}
\newtheorem{corollary}[definition]{Corollary}

\theoremstyle{remark}
\newtheorem{remark}[definition]{Remark}

\usepackage{pgfplots}
\usetikzlibrary{calc}
\pgfplotsset{compat=newest}
\pgfplotsset{plot coordinates/math parser=false}
\newlength\figureheight
\newlength\figurewidth
\usetikzlibrary{arrows,calc,decorations.markings} 
\newcommand{\drawarcdelta}[4]{
  \draw ($#1+(#2:#4)$) arc[start angle=#2, delta angle=#3, radius=#4];
}
\newcommand{\drawarcdeltadotted}[4]{
  \draw [black,thick, dotted]($#1+(#2:#4)$) arc[start angle=#2, delta angle=#3, radius=#4];
}
\newcommand{\drawlabeledarcdelta}[6]{
  \drawarcdelta{#1}{#2}{#3}{#4}
  \node at ($#1+(#2+#3/2:#6)$) {#5};
}

\newcommand{\Bo}{\mathcal{B}}

\newcommand{\C}{\mathbb{C}}

\newcommand{\N}{\mathbb{N}}
\newcommand{\R}{\mathbb{R}}

\newcommand{\D}{\mathbb{D}}
\newcommand{\T}{\mathbb{T}}

\newcommand{\veps}{\varepsilon}

\begin{document}

\title{Functional calculus estimates for {T}admor--{R}itt operators}



\author{Felix L.~Schwenninger}

\address{Felix L.~Schwenninger \\
              University of Twente, \\
              P.O.~Box 217, 7500 AE Enschede\\ 
	      The Netherlands\\
              Tel.: +31-53-4892230}
              \email{f.l.schwenninger@utwente.nl}           %  \\
          % \and
           %S. Author \at
           %   second address

\thanks{The author has been supported by the Netherlands Organisation for Scientific Research (NWO), grant no.~613.001.004. 
% General acknowledgments should be placed at the end of the article.
}

%\normalsize

\begin{abstract}\normalsize
%Insert your abstract here. No references or citations in abstract! 
We show $H^{\infty}$-functional calculus estimates for Tadmor--Ritt operators (also known as Ritt operators), which  generalize and improve results by Vitse. These estimates are in conformity with the best known power-bounds for Tadmor--Ritt operators in terms of the constant dependence. 
Furthermore, it is shown how discrete square function estimates influence the estimates.
%Include keywords and mathematical subject classification numbers as needed.

\end{abstract}


\maketitle
%\def\com{1}  %either 0 (comments off) or 1 (comments on)

%\usepackage{palatino,eulervm}
%Additional packages
%\usepackage{fullpage}
%\usepackage{marginnote}
%\usepackage{showlabels}

%TO DO
%

%2) add clear definition of H^{\infty} on Stolz domains. Maximum principle.
%3) check the published version of Le Merdys paper.
%Add: Continuous case explanation, refer to Sch 15
%Add: References: Tomilov Paper, Markus Transference, alpha admissibility Le Merdy.

\section{Introduction}

When studying  numerical stability of a difference equation of the form 
\begin{equation}\label{eq1}
	x_{n}=Tx_{n-1}+r_{n},\quad n>0,
\end{equation}
the notion of \textit{power-boundedness} emerges naturally.
Here, $T$ is a square matrix or, more general, a linear operator and stability w.r.t.\ to the initial value $x_{0}$ can be measured by
\begin{equation*}
	E=\sup\nolimits_{n\in\N}\|x_{n}-\tilde{x}_{n}\|,
\end{equation*}
where $\tilde{x}_{n}$ denotes the solution to \eqref{eq1} with initial value $\tilde{x}_{0}$. Since $\|x_{n}-\tilde{x}_{n}\|=\|T^{n}(x_{0}-\tilde{x}_{0})\|$, the question whether $E<\infty$ (for all $x_{0}$, $\tilde{x}_{0}$) reduces to asking if $\sup_{n}\|T^{n}\|<\infty$.

Although the characterization of power-bounded matrices in terms of the eigenvalues is well-known, one aims for different conditions implying power-boundedness, like conditions on the resolvent $(zI-T)^{-1}$. The most famous characterization for matrices is probably given by the \textit{Kreiss Matrix Theorem} \cite{Kreiss62,Spijker91}. \newline
As the Kreiss Matrix Theorem fails for infinite dimensions, one has to strengthen the conditions on the resolvent in order to guarantee power-boundedness. This leads to the notion of \textit{Tadmor--Ritt operators}.\newline
This paper deals with general estimates for Tadmor--Ritt operators, which particularly imply power-boundedness.
\medskip

In the following, let $\D$ denote the open unit disc in the complex plane, $\overline{\D}$ its closure and $\partial\D$ its boundary. %$\Cp,\Cm$ will refer to the right and the left half-plane respectively.
%For an operator $A$ on a Banach space, $D(A)$ denotes the domain and $R(A)$ the range of $A$. 
The spectrum of a linear (bounded) operator $T:X\rightarrow X$ on a Banach space $X$ will be denoted by $\sigma(T)$ and the resolvent set by $\rho(T)$.  
For $z\in\rho(T)$, we define the resolvent $R(z,T):=(zI-T)^{-1}$. In the following, all considered operators $T$ are assumed to be linear and bounded.\newline 
For an open set $\Omega\subset\C$, $H^{\infty}(\Omega)$ denotes the space of bounded analytic functions on $\Omega$, equipped with the supremum norm $\|\cdot \|_{\infty,\Omega}$. 

\subsection{Tadmor--Ritt and Kreiss operators}
In the following we will give a brief introduction about Tadmor--Ritt operators, and explain their relation to more general Kreiss operators. Unless stated otherwise, $X$ will denote a general Banach space.
\begin{definition}
An operator $T$ on $X$ is called a \textit{Tadmor--Ritt} operator if $\sigma(T)\subset \overline{\D}$ and if
\begin{equation}\label{eq:TadmorRitt}
	C(T):= \sup_{|z|>1} \| (z-1) R(z,T) \| < \infty.
\end{equation}
Let $TR(X)$ denote the set of all Tadmor--Ritt operators on $X$.
\end{definition}
Tadmor--Ritt operators, in the literature also sometimes referred to as  \textit{Ritt operators}, were, with a slightly different but equivalent definition, first studied in \cite{Ritt53}.
See \cite{Bakaev03,BoroDrissiSpijker,Ransford02,Vitse04Turndown} for a detailed discussion of these two definitions.
Tadmor--Ritt operators form a class consisting of operators satisfying  \textit{Kreiss' resolvent condition},
\begin{equation} \label{eq:Kreisscondition}
\sigma(T)\subset\overline{\D},\quad\text{and}\quad C_{Kreiss}(T)=\sup_{|z|>1}\| (|z|-1) R(z,T) \|< \infty.
\end{equation}
We will call operators satisfying (\ref{eq:Kreisscondition}) \textit{Kreiss operators} and denote the set of all such operators on $X$ by $KR(X)$. 
Obviously, $TR(X)\subset KR(X)$.
The most prominent question related to these operators is the one of \textit{power-boundedness}, i.e.\ whether 
\begin{equation*}
Pb(T):=\sup_{n\in\N}\| T^{n}\| <\infty.
\end{equation*}
%where we follow the notation from \cite{Nikolski14}.
%Originally, %and of more importance for numerical analysis, 
In 1962, O.~Kreiss studied the question for finite-dimensional spaces $X$, \cite{Kreiss62}.
He showed that  in this case the answer is positive and that for all $T\in KR(X)$,
\begin{equation}\label{eq:generalKMT}
	Pb(T) \leq g(C_{Kreiss}(T),N),
\end{equation}
for a function $g$ depending on $C_{Kreiss}(T)$ and the dimension $N$ of the space $X$. 
Kreiss' originial estimate (of the function $g$) was improved steadily in the following decades ending up with the final result proved by Spijker in 1991, \cite{Spijker91},
\begin{equation}\label{eq:SpijkerKMT}
 \forall T\in KR(X):\quad 	Pb(T) \leq e C_{Kreiss}(T) N.
\end{equation}
For the detailed history of the result we refer to the monograph \cite{TrefethenEmbree05} and the recent work \cite{Nikolski14}. 
By \cite{LeVequeTrefethen84}, estimate (\ref{eq:SpijkerKMT}) is sharp in the sense that there exists a sequence of matrices $T_{N}\in KR(\C^{N\times N})$ such that 
\begin{equation*}
	\lim_{N\to\infty}\frac{Pb(T_{N})}{C_{Kreiss}(T_{N})N}=e.
\end{equation*}
However, for this sequence, $C_{Kreiss}(T_{N})\to\infty$, hence, for $C_{Kreiss}(T)\leq C$ with a fixed constant $C$, the behavior could theoretically be better.
Indeed, a recent result by Nikolski shows that for $T$ having \textit{unimodular} spectrum, i.e.\ $\sigma(T)\subset\partial\D$, and a basis of eigenvectors, one gets a \textit{sublinear} growth in the dimension.

\begin{theorem}[N.~Nikolski 2013 \cite{Nikolski14}]\label{thmNikolski}
Let $X$ be a Hilbert space of dimension $N<\infty$. Let $T$ be a Kreiss operator on $X$ such that $\sigma(T)\subset\partial\D$ and such that $T$ has a basis of eigenvectors $\mathcal{X}_{N}=(x_{j})_{j=1}^{N}$. Then
\begin{equation*}
Pb(T) \leq 2\pi C_{Kreiss}(T)N^{1-\veps},
\end{equation*}
where $\veps=\frac{0.32}{b(\mathcal{X}_{N})^{2}}$ and $b(\mathcal{X}_{N})$ denotes the basis constant of $\mathcal{X}_{N}$, i.e.
\begin{equation*}
b(\mathcal{X}_{N})=\sup_{k\leq l}\| P_{(x_{j})_{j=l}^{k}} \|,
\end{equation*}
where $P_{(x_{j})_{j=l}^{k}}$ denotes the projection onto the span of the vectors $(x_{j})_{j=l}^{k}$.
\end{theorem}
The proof of Theorem \ref{thmNikolski} is based on a classic theorem by McCarthy and Schwartz \cite{McCarthySchwartz65}.
We remark that Nikolski also shows a corresponding result on more general Banach spaces using a generalization of McCarthy and Schwartz' result by Gurari and Gurari \cite{Gurarii71}.
By using well-known techniques from Spijker, Tracogna, Welfert \cite{SpijkerTracognaWelfert03}, he further proves that the sublinear behavior is sharp. 
As indicated by Nikolski, in order to get an estimate in the spirit of the Kreiss Matrix Theorem, one has to close the loop by estimating $b(\mathcal{X}_{N})$ in terms of $C_{Kreiss}(T)$. This still remains open.

If we turn to general infinite-dimensional spaces $X$, the power-boundedness of general Kreiss operators, even on Hilbert spaces, is no longer true. 
We refer to \cite{Foguel64} and \cite{Halmos64} for counterexamples. 
In the conference paper \cite{TadmorLAA}, E.~Tadmor states that the growth of $\|T^{n}\|$ can at most be logarithmically in $n$ 
under the additional assumption that the spectrum of $T$ `is not too dense in the neighbourhood of the unit circle'.  This condition is in particular ensured if (\ref{eq:TadmorRitt}) holds. 
Moreover, the existence of an example is stated confirming the sharpness of the growth. 
As both the proof and the example are unfortunately not published, we are indebted to E.~Tadmor for sharing them with us, \cite{Tadmor2014private}. \\
Knowing that general Kreiss operators are not power-bounded, the same question for Tadmor--Ritt operators remained open until 1999 when  Lyubich, \cite{Lyubich99}, and Nagy $\&$ Zemanek, \cite{NagyZemanek99} 
used a preceding result of O.~Nevanlinna, \cite{Nevanlinna93}, to prove that they are indeed power-bounded. 
We remark that in 1993, C.~Palencia \cite{Palencia93} and, independently, Crouzeix, Larsson, Piskarev and Thom\'{e}e \cite{Crouzeixetal} showed that the \textit{Crank--Nicolson-scheme} is stable for sectorial operators. In particular, this shows that the \textit{Cayley transform} $Cay(A):=(I-A)(A+I)^{-1}$ of a sectorial operator $A$ is power-bounded. As it well-known that the mapping $A\mapsto Cay(A)$ establishes a one-to-one correspondence between sectorial operators $A$ with $0\in\rho(A)$ and Tadmor--Ritt operators, the result already shows the power-boundedness of Tadmor--Ritt operators. This fact seems to be unnoticed in the literature. Moreover, Palencia's result shows that any bounded operator $S$ with $\sigma(S)\subset\overline{\D}$ and such that there exists a constant $M(S)>0$ and
\begin{equation*}
	\|R(z,S)\| \leq  M(S)( |z+1|^{-1} + |z-1|^{-1}) , \quad |z|>1,
\end{equation*}
is power-bounded. Note that Tadmor--Ritt operators are of this form.

In 2002, El-Fallah and Ransford, \cite{Ransford02} showed that for a Tadmor--Ritt operator $T$, $Pb(T)\leq C(T)^{2}$,  which was subsequently improved  by Bakaev \cite{Bakaev03} to
\begin{equation}\label{eq:Bakaev}
	\forall T\in TR(X):\quad 	Pb(T) \leq aC(T)\log(aC(T)),
\end{equation}
 for some absolute constant $a>0$ (which was not determined). The latter result seems to be not so well-known. 
In \cite[Remark 2.2]{Vitse2005b} an alternative proof for the quadratic dependence on $C(T)$ is sketched. 
 A careful study of this sketch reveals  that it is based on a similar approach as in Bakaev's proof, which, with a sharper estimation and some additional work, actually yields  (\ref{eq:Bakaev}).
We will encounter a similar approach in the proof of Theorem \ref{thm1}, which was actually motivated by a result of the author for analytic semigroups, \cite{Schwenninger15a}.
\medskip

In \cite{Vitse04,Vitse2005b} Vitse investigated the more general setting of a functional calculus for Tadmor--Ritt operators and proved that for $1\leq m\leq n$ and any polynomial $p(z)=\sum_{j=m}^{n}a_{j}z^{j}$
\begin{equation}\label{eq:Vitse1}
		\| p(T) \| \leq c(C(T),m,n)\cdot\sup_{z\in\D}|p(z)|,
\end{equation}
with $c(C(T),m,n)= 191 C(T)^{5} \log\left(\frac{e(n+1)}{m}\right)$. 
We also remark that Le Merdy showed in \cite{LeMerdy98} that a Tadmor--Ritt operator on a Hilbert space has bounded polynomially calculus, i.e.\
\begin{equation}\label{eq:Polybdd}
\sup\left\{\|p(T)\|:p \text{ is polynomial},\|p\|_{\infty,\D}\leq1\right\}<\infty,
\end{equation}
 if and only if $T$ is similar to a contraction. 
Obviously, (\ref{eq:Vitse1})  implies power\allowbreak-bounded\-ness of $T$, however, yet with a $C(T)$-dependence worse than in (\ref{eq:Bakaev}). \newline
By functional calculus, more general functions $f$ in $H^{\infty}(\D)$ can be considered instead of polynomials $p$ in \eqref{eq:Vitse1}. This leads to the study of the \textit{$H^{\infty}$-calculus} for Tadmor--Ritt operators \cite{ArhancetFacklerLeMerdy2015,ArhancetLeMerdy2014,LancienLeMerdy2013}.
\medskip

We will show that the constant $c(C(T),m,n)$ in \eqref{eq:Vitse1} can be improved significantly, coupling it to the, so-far known, optimal constant for the power-bound of $T$ in (\ref{eq:Bakaev}).
Precisely, in Theorem \ref{thm2} we will show that for $p(z)=\sum_{j=m}^{n}a_{j}z^{j}$, $0\leq m\leq n$,
\begin{equation}\label{eq:mainresult}
	\|p(T)\| \leq aC(T)\log\left(C(T)+b+\log\frac{n+1}{m+1}\right) \cdot \|p\|_{\infty,\D}, 
\end{equation}
with absolute constants $a,b>0$. The proof is shorter and more direct than the one for (\ref{eq:Vitse1}) in \cite{Vitse05}. Note also that we allow for $m=0$.\newline
Moreover, the result is actually a consequence of a more general functional calculus result for Tadmor--Ritt operators, 
see Theorem \ref{thm1}. 
%As a direct consequence, our results improve the constant dependence on $C(T)$ for the Besov-calculus derived by Vitse in \cite{Vitse04,Vitse2005b},
%see Theorem \ref{thm4}.

Finally, motivated by the result for analytic semigroup generators (i.e.\ sectorial operators) \cite{Schwenninger15a}, which can be seen as the continuous counterparts of Tadmor--Ritt operators,
we discuss the influence of \textit{square function estimates} on the the calculus estimates, see also \cite{LeMerdy2014a}. For Hilbert spaces, it is known that if a Tadmor--Ritt operator and its dual operator satisfy square function estimates, then the corresponding $H^{\infty}$-functional calculus is bounded. As for the more known continuous counterpart of sectorial operators, here, it is essential to have square function estimates for both $T$ and $T^{*}$. We show that having only $T$ (or alternatively $T^{*}$) satisfying square function estimates however improves the functional calculus estimate \eqref{eq:mainresult}, see Theorem \ref{thm3}. 
In Section \ref{subsec:GeneralSQFE}, we generalize the result about square function estimates to general Banach spaces. This involves a refined definition of square function estimates using Rademacher means and \textit{$R$-boundedness}. These abstract square function estimates are the discrete counterpart to the ones for sectorial operators, which were introduced by Kalton and Weis \cite{KaltonWeisUnpublished} and have proved very useful in the study of $L^{p}$-maximal regularity for parabolic evolution equations since then. \newline
In Section \ref{subsec:SharpnessEstimates}, we discuss sharpness of the derived estimates. We conclude by a result about a Besov-space calculus for Tadmor--Ritt operators, which is a refinement of \cite[Theorem 2.5]{Vitse2005b}.
\subsection{Properties of Tadmor--Ritt operators}\label{sec:PropTR}
Unless stated otherwise, $X$ will always denote a, in general infinite-dimensional, Banach space.\newline
From (\ref{eq:TadmorRitt}) it follows that for Tadmor--Ritt operators the  only possible spectral  point on $\T$ is $1$. 
Moreover, it is well-known that the spectrum is contained in  the \textit{Stolz type} domain $\mathcal{B}_{\theta}$, which is the interior of the convex hull of 
$\left\{\left\{1\right\},B_{\sin\theta}(0)\right\}$ for some $\theta\in(0,\frac{\pi}{2})$, see Figure \ref{fig1}. 
Here, $B_{r}(z_{0})$ denotes the open ball  centred at $z_{0}$ with radius $r$. 
For this and a proof of the following lemma we refer to Vitse \cite{Vitse04Turndown,Vitse2005b} and Le Merdy \cite{LeMerdy2014a},
which improves earlier results in \cite{Lyubich99,NagyZemanek99} and \cite{Nevanlinna93}.
\begin{lemma}%[Lemma 1.2 in \cite{Vitse04Turndown}]
	\label{Lemma1}
	Let $T$ be a Tadmor--Ritt operator on a Banach space $X$. Then, there exists $\theta\in[0,\frac{\pi}{2})$ such that
	\begin{enumerate}[label={(\roman*)}]
		\item\label{Lemma1it1}  $\sigma(T)\subset \overline{\mathcal{B}_{\theta}}$, and
		\item\label{Lemma1it2}for all $\eta\in(\theta,\frac{\pi}{2}]$,
		 \begin{equation} \label{eq:TadmorRitttheta}
			C_{\eta}(T)=\sup\nolimits_{z\in\C\setminus\overline{\mathcal{B}_{\eta}}} \|(z-1)R(z,T)\| \leq \frac{C(T)}{1-\frac{\cos{\eta}}{\cos\theta}}.
		\end{equation}
	\end{enumerate}
	We say that $T$ is \textit{of type $\theta$}. \newline
	Moreover, $\theta$ can always chosen to be $\theta=\arccos\frac{1}{C(T)}$.
\end{lemma} 
Note that $\mathcal{B}_{\alpha}\subset \mathcal{B}_{\beta}$ for $\alpha<\beta$.	
The previous lemma tells us that for $\eta$ going to $\theta$, the right-hand-side of $\eqref{eq:TadmorRitttheta}$ explodes 
whereas for $\eta=\frac{\pi}{2}$ it becomes $C(T)$. We further remark that the converse of Lemma \ref{Lemma1} also holds: If there exists $\theta\in(0,\frac{\pi}{2})$ such that $\sigma(T)\subset\overline{\mathcal{B}_{\theta}}$ and $C_{\eta}<\infty$ for all $\eta\in(\theta,\frac{\pi}{2})$, then $T$ is Tadmor--Ritt, see \cite[Lemma 2.1]{LeMerdy2014a}.\newline
We further need the following well-known characterization, which can be found e.g., in \cite{LeMerdy2014a,Lyubich99,NagyZemanek99,Vitse2005b}.
\begin{lemma}\label{lem:charRitt}
Let $T$ be an operator on a Banach space $X$. The following  assertions are equivalent.
\begin{enumerate}[label=\textit{(\roman*)}]
\item $T$ is Tadmor--Ritt.
\item \label{it2lemcharRitt} The sets $\left\{T^{n}:n\in\N\right\}$ and $\left\{n(T^{n}-T^{n-1}):n\in\N\right\}$ are bounded, i.e.
		\begin{equation}\label{eq1:lemcharRitt}
		Pb(T)=\sup_{n\in\N}\|T^{n}\|<\infty, \text{ and } c_{1,T}:=\sup_{n\in\N}\|n(T^{n}-T^{n-1})\|<\infty.
		\end{equation}
\end{enumerate}
%The equivalence remains true, if `bounded' gets replaced by `R-bounded' and `Tadmor--Ritt' by `$R$-Ritt', see \cite{LeMerdy2014a} and Section \ref{subsec:GeneralSQFE} for the definitions.
\end{lemma}
Let us emphasize that $\sup_{n\in\N}\|n(T^{n}-T^{n-1})\|<\infty$ does not imply power-bounded\-ness of $T$ in general, see \cite{KaltonTomilov}. Hence, in \ref{it2lemcharRitt} of Lemma \ref{lem:charRitt}, the assumption of power-boundedness cannot be dropped. See also \cite{Nevanlinna94} for a discussion on power-boundedness related to estimates on the resolvent.

%\section{Main results}

\section{A functional calclulus result for Tadmor--Ritt operators}

 By Lemma \ref{Lemma1} we know that the spectrum of a Tadmor--Ritt operator is contained in the Stolz type domain $\overline{\mathcal{B}_{\theta}}$,
with $\theta=\arccos\frac{1}{C(T)}$. Let $\Omega\supset\overline{\mathcal{B}_{\theta}}$ be an open, bounded and simply connected subset of $\C$.
Then for any function holomorphic on  $\Omega$, the 
operator $f(T)$ can be defined via the Riesz-Dunford integral
	\begin{equation}\label{eq:RieszDunford}
		f(T)=\frac{1}{2\pi i}\int_{\Gamma} f(z)\ R(z,T) \ dz,
	\end{equation} 
	where $\Gamma$ is a rectifiable, positively orientated, simple contour inside $\Omega$ which encircles $\overline{\mathcal{B}_{\theta}}$. Let $H^{\infty}(\Omega)$ denote the bounded holomorphic functions on $\Omega$. 
	\begin{remark}\label{rem21} Let $H_{0}^{\infty}(\mathcal{B}_{\delta})$ be the functions $f$ in $H^{\infty}(\mathcal{B}_{\delta})$ for which exist constants $c,s>0$ such that $f(z)\leq c|1-z|^{s}$ for all $z\in\mathcal{B}_{\delta}$. For $\delta\in(\theta,\frac{\pi}{2})$ and $f\in H_{0}^{\infty}(\mathcal{B}_{\delta})$, $f(T)$ can still be defined by (\ref{eq:RieszDunford}) with $\Gamma$ equal to the boundary of $\partial\mathcal{B}_{\delta'}$ of $\mathcal{B}_{\delta'}$ with $\delta'\in(\theta,\delta)$. Analogously to the situation for sectorial operators, see e.g., \cite{haasesectorial}, it can be shown that the mapping $f\mapsto f(A)$ becomes an algebra homomorphism from $H_{0}^{\infty}(\mathcal{B}_{\delta})$ to $\mathcal{B}(X)$, see \cite[Section 2]{LeMerdy2014a} for more details. 
	\end{remark}
	For $0<r<1$ and $\eta\in(0,\frac{\pi}{2}]$, we define the `keyhole-shaped' set, 
		\begin{equation}\label{eq:Omega}
			\Omega_{\eta,r}:= \mathcal{B}_{\eta} \cup B_{r}(1),
		\end{equation}
		see Figure \ref{fig1}.
	\begin{figure}
		\label{fig1}	\begin{tikzpicture}[scale=2.1]
\normalsize 
% To change the font size in tikz picture
%\draw [blue] (0,0) rectangle (1.5,1);
%\draw [red, ultra thick] (3,0.5) circle [radius=0.5];;

%\draw[help lines] (1/-1.4142,1/-1.4142) grid (3/1.4142,3/1.4142);

\pgfmathsetmacro{\THETA}{30}
\pgfmathsetmacro{\ETA}{40}
\pgfmathsetmacro{\R}{0}
\pgfmathsetmacro{\COSMINR}{cos(\ETA)}
\pgfmathsetmacro{\SINTHETA}{sin(\THETA)}
\pgfmathsetmacro{\SINETA}{sin(\ETA)}

\coordinate (zero) at (0,0);
\coordinate (one) at (1,0);
\coordinate (imag) at (0,1);
\draw [black, very thick,fill=lightgray, fill opacity=.7] ([shift=(60:\SINTHETA)] 0, 0) arc [radius=\SINTHETA, start angle={90-\THETA}, end angle={270+\THETA}];

%\pgfmathsetlengthmacro{\radius}{}
\drawarcdelta{(zero)}{60}{240}{\SINTHETA}

\drawarcdelta{(zero)}{0}{360}{1}
\drawarcdeltadotted{(one)}{140}{-280}{\R}
%\draw [black,very thick,dotted] (one)+(140:\R);

\draw [black] (zero) -- (60:\SINTHETA);
\draw [black, very  thick,fill=lightgray, fill opacity=.7]  (one)+(150:{cos(\THETA}) --(one)-- +(-150:{cos(\THETA});
%\draw [black, very thick] (0,-1.4142) -- (1/1.4142,1/-1.4142);
%\draw [black] (0,-1.4142) -- (1/-1.4142,1/-1.4142);
% [lightgray, fill=lightgray] (1/-1.4142,1/-1.4142) -- (4+1/-1.4142,3+1/-1.4142) -- (4+1/-1.4142,-3+1/-1.4142);

%angle labelling
\drawlabeledarcdelta{(60:\SINTHETA)}{-120}{90}{0.15}{$\cdot$}{0.08}

\drawlabeledarcdelta{(one)}{150}{30}{0.5}{\small$\theta$}{0.35}
%x and y axis.
\draw [<->,black] (-1.2,0) -- (1.2,0);
\draw [<->,black] (0,-1.2) -- (0,1.2);

%sigma(T)
\node [below left] at ([shift=(225:{\SINTHETA*0.1})] zero) {$\mathcal{B}_{\theta}$};

%\node [above right] at (1/1.4142,1/-09.4142) {$\partial\mathbb{D}$};
\node [below right] at (zero)  {\small$0$};
\node [below right] at (one) {\small$1$};
\node [above right] at (imag) {\small$i$};
\normalsize
\end{tikzpicture}
\hspace{1cm}
	\begin{tikzpicture}[scale=2.1]
\normalsize 
\pgfmathsetmacro{\THETA}{30}
\pgfmathsetmacro{\ETA}{40}
\pgfmathsetmacro{\R}{0.3}
\pgfmathsetmacro{\COSMINR}{cos(\ETA)}
\pgfmathsetmacro{\SINTHETA}{sin(\THETA)}
\pgfmathsetmacro{\SINETA}{sin(\ETA)}

\coordinate (zero) at (0,0);
\coordinate (one) at (1,0);
\coordinate (imag) at (0,1);

\draw [black, thick] ([shift=(60:\SINTHETA)] 0, 0) arc [radius=\SINTHETA, start angle={90-\THETA}, end angle={270+\THETA}];

%\pgfmathsetlengthmacro{\radius}{}
\drawarcdelta{(zero)}{60}{240}{\SINTHETA}
\drawarcdeltadotted{(zero)}{50}{260}{\SINETA}
\drawarcdelta{(zero)}{0}{360}{1}
\drawarcdeltadotted{(one)}{140}{-280}{\R}
%\draw [black,very thick,dotted] (one)+(140:\R);

\draw [->,black, very thick, dashed] (one)+(140:\R) -- +(140:\COSMINR);
\draw [<-,black, very thick, dashed] (one)+(-140:\R) -- +(-140:\COSMINR);
%\draw [black] (0,0) -- (1/-1.4142,1/-1.4142);
\draw [black, thick]  (one)+(150:{cos(\THETA}) --(one)-- +(-150:{cos(\THETA});
\drawlabeledarcdelta{(one)}{140}{40}{0.6}{$\eta$}{0.45}

%x and y axis.
\draw [<->,black] (-1.2,0) -- (1.2,0);
\draw [<->,black] (0,-1.2) -- (0,1.2);

%sigma(T)
\node [below left] at ([shift=(225:{\SINTHETA*0.1})] zero) {$\partial\mathcal{B}_{\theta}$};

\draw [->, black] (one) -- node [above]{\small$r$}+(20:\R);

%\node [above right] at (1/1.4142,1/-09.4142) {$\partial\mathbb{D}$};
\node [below right] at (zero)  {\small$0$};
\node [below right] at ([shift=(-90:\R)] one)  {$\partial\Omega_{\eta,r}$};
\node [below right] at (one) {\small$1$};
\node [above right] at (imag) {\small$i$};
%angle
%\draw [black, thick] (-0.6+1/1.4142,1/-1.4142) arc [radius=0.6, start angle=180, end angle= 135];
%\draw [black, thick] (0,0) arc [radius=0.6, start angle=225, end angle= 315];
%\node [above] at (-0.4+1/1.4142,1/-1.4142) {\small $\theta$};
\normalsize
\end{tikzpicture}
	\caption{The sets $\mathcal{B}_{\theta}$ and $\Omega_{\eta,r}$ with $\eta\in\left(\theta,\frac{\pi}{2}\right)$.}
		\end{figure}
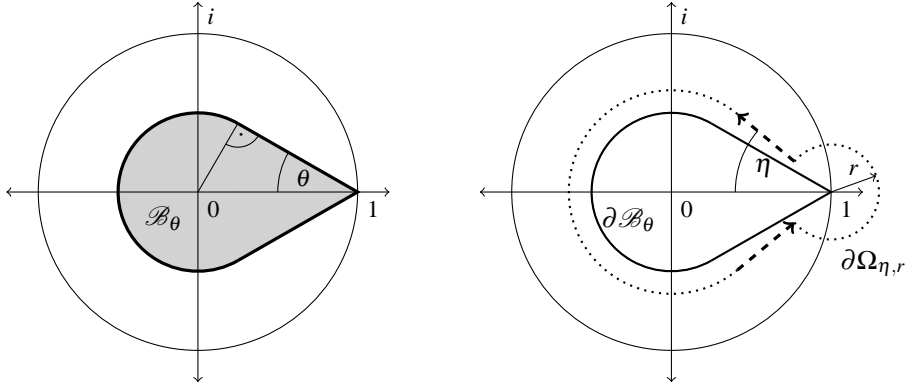
	\newline The function 
	\begin{equation}\label{eq:EI}
	{\rm Ei}(s)=\int_{s}^{\infty}\frac{e^{-x}}{x}dx
	\end{equation}
	 is known as the \textit{Exponential integral}. It holds that 
	\begin{align}\label{est:Ei}
	\tfrac{1}{2}e^{-s}\log\left(1+\tfrac{2}{s}\right)<{}&{\rm Ei}(s)< e^{-s}\log(1+\tfrac{1}{s}), \ s>0,\\
	&{\rm Ei}(s)<\log(\tfrac{1}{s}),\ s\in(0,\tfrac{1}{2}],\label{est2:Ei}
	\end{align}
	see \cite{Gautschi59} and \cite{Schwenninger15a} for more details. \newline
	The following lemma outsources technicalities in the proof of the results to come, Theorem \ref{thm1}. 
	Estimates of this kind for deriving functional calculus estimates can already be found in \cite{Bakaev03,Palencia93,Vitse2005b}, see also \cite{LeRoux,LubichNevanlinna} for a slightly different setting.  Here, the focus is laid on deriving estimates explicitly in the used constants.
	\begin{lemma}\label{lem2}
		For $0<r<1$ , $m\geq0$ and $\eta\in(0,\frac{\pi}{2})$, we have that
			\begin{equation}\label{eq:lem2}
			 G(m,\eta,r) := \int_{\partial\Omega_{\eta,r}} \frac{|z|^{m}}{ |z-1|} \ |dz| \leq \mathcal{C}(r,m,\eta)
			\end{equation}
			where $\partial\Omega_{\eta,r}$ denotes the boundary of the set defined in (\ref{eq:Omega}) and
			\begin{equation*}
			\mathcal{C}(r,m,\eta):=4(\sin\eta)^{m+1}\log\tfrac{4}{\cos\eta}+4{\rm Ei}\left(r\tfrac{m+1}{2}\cos\eta\right)+2\pi(1+r)^{m},
			%\left\{\begin{array}{c} 4(\sin\eta)^{m+1}\log\tfrac{4}{\cos\eta}+{\rm Ei}\left(\frac{r(m+1)\cos\eta}{2}\right)+2\pi(1+r)^{m}, & m>0,\\4\sin\eta\log\tfrac{4}{\cos\eta}+\left|\log\frac{\cos\eta}{r}\right|+2\pi,& m=0,
			%\end{array}\right.
			\end{equation*}
			where $\rm Ei$ is defined in (\ref{eq:EI}). \newline If $r\leq\frac{1}{m+1}$, by (\ref{est2:Ei}), 
			\begin{equation*}
			\mathcal{C}(r,m,\eta)\leq -8\log\cos\eta-4\log(r(m+1))+2\pi(1+r)^{m}+12\log2.
			\end{equation*}
	\end{lemma}
	\begin{proof}
	Let us first assume that $r<\cos\eta$. We split up the path $\partial\Omega_{\eta,r}=\Gamma_{1}\cup\Gamma_{2}\cup\Gamma_{3}$, where $\Gamma_{2}$ denotes the union of the two straight line segments of $\partial\Omega_{\eta,r}$ (dashed lines in Figure \ref{fig1}), whereas $\Gamma_{1}$, $\Gamma_{3}$ denote the  part of $\partial\Omega_{\eta ,r}$ that lies
 	 on the circles $B_{\sin\eta }(0)$ and $B_{r}(1)$, respectively (dotted lines in Figure \ref{fig1}). Precisely, 
	\begin{align*}
			\Gamma_{1} =&\left\{(\sin\eta)  e^{i\delta},|\delta|\in(\tfrac{\pi}{2}-\eta ,\pi]\right\},
			\Gamma_{2}=\left\{1-te^{\pm i\eta },t\in(r,\cos\eta ]\right\},\\
			\Gamma_{3}=&\left\{1+re^{i\delta}, |\delta|\in[0,\pi-\eta )\right\},
	\end{align*} 
	Next we estimate $G_{i}:=\int_{\Gamma_{i}}\frac{|z|^{m}}{|z-1|}|dz|$ for $i=1,2,3$.\newline
	For $\Gamma_{1}$, we see that 
	\begin{equation*}
		G_{1} = 2(\sin\eta)^{m+1} \int_{\frac{\pi}{2}-\eta}^{\pi} \frac{dx}{| e^{ix}\sin\eta-1|}.
	\end{equation*}
	Since $2|Re^{ix}-1|\geq |e^{ix}-1|$ for all $R,x\geq0$, 
	\begin{equation*}
		G_{1}\leq 4(\sin\eta)^{m+1}\int_{\frac{\pi}{2}-\eta}^{\pi}\frac{dx}{|e^{ix}-1|}=2\sqrt{2}(\sin\eta)^{m+1}\int_{\frac{\pi}{2}-\eta}^{\pi}\frac{dx}{\sqrt{1-\cos x}}.
	\end{equation*}
	Since $\sqrt{2}\log\tan\frac{x}{4}$ is a primitive of $\frac{1}{\sqrt{1-\cos x}}$ for $x\in(0,\pi)$, we derive
	\begin{equation*}
		G_{1}\leq-4(\sin\eta)^{m+1}\log\tan\frac{\frac{\pi}{2}-\eta}{4} \leq4(\sin\eta)^{m+1}\log\tfrac{4}{\frac{\pi}{2}-\eta}, 
	\end{equation*}
	where in the last step we used that $\tan x\geq x$ for $x\in[0,\frac{\pi}{4}]$, which follows from the Taylor series of $\tan$. Since $\sin x\leq x$ for all $x\geq0$, we finally get
	\begin{equation}\label{eq:estG1}
	G_{1}\leq 4(\sin\eta)^{m+1}\log\frac{4}{\cos\eta}.
	\end{equation}
	\newline
%	Alternatively, one can estimate as follows. By the identity $\sin 2y=2\sin y \cos y$, we get
%	\begin{equation*}
%	\sin\tfrac{x}{4}=\frac{\sin x}{4\cos(\tfrac{x}{2})\cos(\tfrac{x}{4})},\ x\in(0,\tfrac{\pi}{2}).
%	\end{equation*}
%	With $x=\frac{\pi}{2}-\eta$, we obtain
%	\begin{equation*}
%	-\log\tan\frac{\frac{\pi}{2}-\eta}{4}\leq\left[\log(\tfrac{4}{\sin x})+2\log(\cos \tfrac{x}{4})+\log(\cos\tfrac{x}{2})\right]\leq\log(\tfrac{4}{\sin x}).
%	\end{equation*}
%	Hence,  the last estimate in (\ref{eq:estG1}) gets replaced by $G_{1}\leq4(\sin\eta)^{m+1}\log(\tfrac{4}{\cos\eta})$.\newline
	To estimate $G_{2}$, note that $|1-te^{i\eta}|^{2}\leq(1-t\cos\eta)$ for $t\in[0,\cos\eta]$ and thus,
	\begin{align*}
	G_{2}={}&\int_{\Gamma_{2}}\frac{|z|^{m}|dz|}{|z-1|}=2\int_{r}^{\cos\eta}|1-te^{i\eta}|^{m}\ \frac{dt}{t}
	\leq2\int_{r}^{\cos\eta}|1-t\cos\eta|^{\frac{m}{2}}\ \frac{dt}{t}
	\end{align*}
	%If $m=0$, then $G_{2}\leq\log\frac{\cos\eta}{r}$. Otherwise further estimations yield
	Since $1-x\leq e^{-x}$ and $e^{-\frac{x}{2}}e^{\frac{1}{2}}\geq1$ for $x\in[0,1]$,
	\begin{align*}
	G_{2}\leq{}&2e^{\frac{1}{2}}\int_{r\cos\eta}^{\cos^{2}\eta}e^{-x\frac{m}{2}-\frac{x}{2}}\frac{dx}{x} \leq 2e^{\frac{1}{2}}\int_{r\frac{m+1}{2}\cos\eta}^{\infty}\frac{e^{-x}}{x}\ dx=4{\rm Ei}\left(r\tfrac{m+1}{2}\cos\eta\right).
	\end{align*}
	Finally, $G_{3}$ can be estimated by
	\begin{equation}\label{eq:estG3}
		G_{3}=2\int_{0}^{\pi-\eta}|1+re^{i\delta}|^{m}\ d\delta\leq2\pi(1+r)^{m}.
	\end{equation}
	This shows (\ref{eq:lem2}) for $r<\cos\eta$.
	
	\smallskip
	If $r\geq\cos\eta $, then $\partial\Omega_{\eta,r}=\partial(B_{\sin\eta}(0)\cup B_{r}(1))$. Hence, we choose $\Gamma_{1}$ and $\Gamma_{3}$ to be the convenient parts of the circles $\partial B_{\sin\eta }(0)$, $\partial B_{r}(1)$ such that $\Gamma_{1}\cup\Gamma_{3}=\partial\Omega_{\eta,r}$, i.e. 
	$\Gamma_{1}=	\left\{\sin\eta  e^{i\delta},|\delta|\in(\alpha ,\pi]\right\}$ and $\Gamma_{3}=\left\{1+re^{i\delta}, |\delta|\in[0,\pi-\beta )\right\}$ for  certain angles $\alpha$, $\beta$ depending on $\eta$.
	Since $r\geq\cos\eta$ it is easy to see that $\alpha>\frac{\pi}{2}-\eta$ and hence, we can estimate similarly as in (\ref{eq:estG1}) and (\ref{eq:estG3}),
	\begin{equation*}
	G(m,\eta,r)=\int_{\Gamma_{1}}+\int_{\Gamma_{3}}\leq G_{1}+2\pi(1+r)^{m}\leq4(\sin\eta)^{m+1}\log\tfrac{4}{\cos\eta}+2\pi(1+r)^{m},
	\end{equation*}
	which concludes the proof as the right hand side is smaller than $\mathcal{C}(r,m,\eta)$.
	\end{proof}
	\begin{theorem}\label{thm1}
	Let $T$ be a Tadmor--Ritt operator on $X$. Let $\theta=\arccos\frac{1}{C(T)}$.
	Then, for $m\in\N_{0}$, $r\in(0,1)$ and $\eta\in(\theta,\frac{\pi}{2})$ we have, with $\tau_{m}(z)=z^{m}$, that
	\begin{equation}\label{eq:thm1}
	\| (f\cdot\tau_{m})(T) \| \leq c(T,m,r,\eta) \cdot \| f \|_{\infty,\Omega_{\eta,r}}
	\end{equation}	
	for $f\in H^{\infty}(\Omega_{\eta,r})$. Here, 
	\begin{equation*}
		c(T,m,r,\eta)\leq \frac{C_{\eta}(T)}{2\pi} \mathcal{C}(r,m,\eta) 
	\end{equation*}
	where $C_{\eta}(T)=\sup_{z\in\C\setminus\overline{\mathcal{B}_{\eta}}} \|(z-1)R(z,T)\|$, and $\mathcal{C}$ as in Lemma \ref{lem2}.
\end{theorem}
\begin{proof}
	Let $\eta\in(\theta,\frac{\pi}{2})$ and $r>0$. By Lemma \ref{Lemma1} we know that $\sigma(T)\subset\Omega_{\eta,r}$. 
	Let $f\in H(\Omega_{\eta,r})$. 
	Since $f\tau_{m}$ is holomorphic on $\Omega_{\eta,r}$,
	\begin{equation*}
		 (f\tau_{m})(T)   = \frac{1}{2\pi i} \int_{\partial \Omega_{\tilde{\eta},\tilde{r}}} f(z) z^{m} R(z,T) \ dz,
	\end{equation*}
	where $\tilde{\eta} \in(\theta,\eta)$ and $\tilde{r}\in(0,r)$.
	Since $\Omega_{\tilde{\eta},\tilde{r}}\subset\Omega_{\eta,r}$, 
	\begin{align*}
		\| (f\tau_{m})(T)  \| \leq {}&\frac{C_{\tilde{\eta} }(T)}{2\pi}\ \|f\|_{\infty,\Omega_{\eta,r}}  \ \int_{\partial \Omega_{\tilde{\eta},\tilde{r}}} \frac{|z|^{m}}{ |z-1|} \ dz\\
		\leq{}&\frac{C_{\tilde{\eta} }(T)}{2\pi}\ \|f\|_{\infty,\Omega_{\eta,r}}\cdot \mathcal{C}(\tilde{r},m,\tilde{\eta}).
	\end{align*}
	 The last inequality followed by Lemma \ref{lem2}. Letting $(\tilde{\eta} ,\tilde{r})\to(\eta,r)$ yields that $C_{\tilde\eta}(T)\allowbreak\to C_{\eta}(T)$ by the maximum principle (applied to $z\mapsto \langle x', zR(z,T)x\rangle$ for $x\in X$, $x'\in X'$) and that $\mathcal{C}(\tilde{r},m,\tilde{\eta})\to \mathcal{C}(r,m,\eta)$ since $\mathcal{C}$ is continuous (see Lemma \ref{lem2}). Together, this gives the assertion.
\end{proof}
%\begin{remark}
%Note that for any $m\in\N$ and $\tau_{m}(z)=z^{m}$, by the maximum principle, the mapping $f\mapsto\tau_{m}f$ is an isometry on $H^{\infty}(\D)$. Therefore, $\|f\|_{\infty,\Omega_{\frac{\pi}{2},r}}\leq r^{m}\|\tau_{m}f\|_{\infty,\Omega_{\frac{\pi}{2}}}$. Let $z^{m}H^{\infty}(\Omega)=\left\{\tau_{m}f:f\in H^{\infty}(\Omega)\right\}$. With $\eta=\arccos\frac{1}{2C(T)}$, Theorem \ref{thm1} shows that 
% $g=\tau_{m}f\in z^{m}H^{\infty}(\Omega_{\frac{\pi}{2},r})$,
%\begin{equation}
%\|g(T)\|\leq C(T,m,r,\eta) \|f\|_{\infty,\Omega_{\eta,r}}\leq c\cdot \|g\|_{\infty,\Omega_{\frac{\pi}{2},r}},
%\end{equation}
%and a constant $c=c(r,m,C(T))>0$. In other words, $T$ admits a bounded $z^{m}H^{\infty}(\D\cup B_{r}(1))$-calculus.
%\end{remark}
The following inequality is a direct consequence of the maximum principle. The disc case ($\eta=\frac{\pi}{2}$) can be traced back to S.~Bernstein, and can be found in \cite[p.~346]{Riesz1916}, or \cite[Problem III. 269, p.137]{PolyaSzegoBook}.%, see also \cite{NewIneqPoly}.
\begin{lemma}\label{lem:Bernstein}
Let $\mathcal{B}_{\alpha}$, $\alpha\in(0,\frac{\pi}{2}]$, be the Stolz type domain defined in Sec.\ \ref{sec:PropTR}. 
The following assertions hold.
\begin{enumerate}[label={(\roman*)},ref={\thelemma~(\roman*)}]
	\item\label{it1:Bernstein} For a polynomial $p$ of degree $n$, and $r\geq1$,
	\begin{equation}\label{eq:Bernstein1}
		\|p\|_{\infty,r\mathcal{B}_{\alpha}} \leq  \left(\frac{r}{\sin\alpha}\right)^{n}\cdot \|p\|_{\infty,\mathcal{B}_{\alpha}}.
	\end{equation}
	\item\label{it2:Bernstein}
	 For $f\in H(\mathcal{B}_{\alpha})$ and continuous on $\overline{\mathcal{B}_{\alpha}}$, $m\in\N$ and $\tau_{m}(z)=z^m$,
	\begin{equation}\label{eq:Bernstein2}
		\|f\cdot\tau_{m}\|_{\infty,\mathcal{B}_{\alpha}}\leq\|f\|_{\infty,\mathcal{B}_{\alpha}}\leq \frac{1}{(\sin\alpha)^{m}}\ \|f\cdot\tau_{m}\|_{\infty,\mathcal{B}_{\alpha}}.
	\end{equation}
\end{enumerate}
\end{lemma}
\begin{proof}
The assertion is a consequence of the maximum  principle applied to $p(z)z^{-n}$. In fact, let $z\in\C\setminus\mathcal{B}_{\alpha}$.
Then, since $z\mapsto p(z)z^{-n}$ is analytic at $\infty$, by the maximum principle,
\begin{equation}\label{eq:Bernstein3}
	|p(z)z^{-n}|\leq \max_{z\in\partial\mathcal{B}_{\alpha}}|p(z)z^{-n}|\leq \max_{z\in\partial\mathcal{B}_{\alpha}}|z^{-n}|\cdot \|p\|_{\infty,\mathcal{B}_{\alpha}}.
\end{equation}
It is easy to see that $\max_{z\in\partial\mathcal{B}_{\alpha}}|z^{-1}|=\frac{1}{\sin\alpha}$. Hence, multiplying \eqref{eq:Bernstein3} by $|z|^{n}$ and noting that $|z|\leq r$ for $z\in \partial(r\mathcal{B}_{\alpha})\subset \C\setminus \mathcal{B}_{\alpha}$ yields
\begin{equation*}
|p(z)|\leq \left(\frac{r}{\sin\alpha}\right)^{n}\|p\|_{\infty,\mathcal{B}_{\alpha}},\ z\in\partial(r\mathcal{B}_{\alpha}).
\end{equation*}
Therefore, \eqref{eq:Bernstein1} follows by the maximum principle. \newline
It is easy to see that  $\sin\alpha\leq|z|$ for $z\in\partial\mathcal{B}_{\alpha}$. Therefore, by the maximum principle,
\begin{equation*}
	\|f\|_{\infty,\mathcal{B}_{\alpha}}=\sup_{z\in\partial\mathcal{B}_{\alpha}}|f(z)| \leq \frac{1}{(\sin\alpha)^{m}}\sup_{z\in\partial\mathcal{B}_{\alpha}}|z^{m}f(z)|=  \frac{1}{(\sin\alpha)^{m}}\|f\tau_{m}\|_{\infty,\mathcal{B}_{\alpha}}
\end{equation*}
The other inequality of \eqref{eq:Bernstein2} is clear as $\mathcal{B}_{\alpha}\subset\D$.
\end{proof}
\begin{theorem}\label{thm2}
	Let $T$  be a Tadmor--Ritt operator on $X$ and let $m,n\in\N$ such that $0\leq m\leq n$.
	 Then, for any $p(z)=\sum_{k=m}^{n}a_{k}z^{k}$, we have that
	\begin{equation}\label{eq:thm2}
		\| p(T) \| \leq aC(T)\left(2\log C(T)+b+\log\frac{n+1}{m+1}\right)  \cdot \| p \|_{\infty,\D},
	\end{equation}
	with absolute constants $a,b$, that can be chosen as
	\begin{equation*}
a=\tfrac{2e}{\pi(1-s)},\quad b=-2\log(s)+6, \quad s\in(0,1).
\end{equation*}	
\end{theorem}
\begin{proof}
 Let $p(z)=\sum_{k=m}^{n}a_{k}z^{k}=z^{m}p_{0}(z)$ with $0\leq m\leq n$ and $p_{0}$ is a polynomial of degree $n-m$. For $s\in(0,1]$ let $\eta(s)=\arccos\frac{s}{C(T)}$. By Theorem \ref{thm1} we have for $s,r\in(0,1)$ that
 \begin{equation}\label{eq:thm2eq1}
 \|p(T)\| \leq c(T,m,r,\eta(s))\cdot \|p_{0}\|_{\infty,\Omega_{\eta(s),r}},
 \end{equation}
 where $p(z)=z^{m}p_{0}$. Since $\Omega_{\eta(s),r}\subset (1+r)\D$, Lemma \ref{it1:Bernstein} (with $\alpha=\frac{\pi}{2}$) yields
 \begin{equation}\label{eq:thm213}
 	\|p_{0}\|_{\infty,\Omega_{\eta(s),r}}\leq\|p_{0}\|_{\infty,(1+r)\D}\leq (1+r)^{n-m}\|p_{0}\|_{\infty,\D}.
 \end{equation}
By the maximum principle, $\|p_{0}\|_{\infty,\D}=\|p\|_{\infty,\D}$. Hence, by choosing $r=\frac{\tau}{n+1}$ with $\tau\in(0,1)$, Eq.\ (\ref{eq:thm2eq1}) becomes
\begin{equation}\label{eq2:thm2}
\|p(T)\| \leq c(T,m,\tfrac{\tau}{n},\eta(s))\cdot (1+\tfrac{\tau}{n+1})^{n-m}\|p\|_{\infty,\D}.
\end{equation}
It remains to estimate the right hand side. Clearly, $(1+\tfrac{\tau}{n+1})^{n-m}\leq e$. 
Theorem \ref{thm1} yields that
\begin{align*}
c(T,m,\tfrac{\tau}{n+1},\eta(s))\leq{}&\frac{C_{\eta(s)}}{2\pi}\mathcal{C}(\tfrac{\tau}{n+1},m,\eta(s)).
\end{align*}
We can further estimate $\mathcal{C}$ using Lemma \ref{lem2}. Since $r=\frac{\tau}{n+1}\leq \frac{1}{m+1}$,
\begin{align*}
c(T,m,\tfrac{\tau}{n+1},\eta(s))\leq{}&\frac{2C_{\eta(s)}}{\pi}(-2\log\cos\eta(s) + \log\frac{n+1}{m+1} -\log\tau + \frac{\pi}{2} e^{\tau} + 3\log2).
\end{align*}
By Lemma \ref{Lemma1}, $C_{\eta(s)}(T)\leq\frac{C(T)}{1-s}$ for $s\in(0,1)$. Since $\cos\eta(s)=\frac{s}{C(T)}$, 
\begin{equation*}
c(T,m,\tfrac{\tau}{n},\eta(s))\leq \frac{2C(T)}{\pi(1-s)}(2\log C(T) - 2\log s + \log\frac{n+1}{m+1} - \log\tau +\frac{\pi}{2} e^{\tau} + 3\log 2).
\end{equation*}
As $\min_{\tau\in(0,1)}\log\frac{1}{\tau}+\frac{\pi}{2} e^{\tau}+3\log2<6$, together with (\ref{eq2:thm2}), this yields (\ref{eq:thm2}). 
\end{proof}
\begin{corollary}\label{cor1}
Let $T$ be a Tadmor--Ritt operator. Then $T$ is power-bounded,
\begin{equation*}
	\sup\nolimits_{n\in\N}\|T^{n}\| \leq aC(T)\left(2\log C(T)+b\right)
\end{equation*}
with absolute constants $a,b>0$ as in Theorem \ref{thm2}.
\end{corollary}
\begin{remark}\label{rem1}\hfill
\begin{enumerate}
\item
Theorem \ref{thm2} shows that a Tadmor--Ritt operator has a bounded $H^{\infty}[m,n]$\allowbreak-calculus, where 
\begin{equation*}
H^{\infty}[m,n]=\{p(z)=\sum\nolimits_{k=m}^{n}a_{k}z^{k}:a_{k}\in\C\}
\end{equation*}
 and $m\leq n$. With different techniques, such a result was proved by Vitse in \cite{Vitse2005b}, see also (\ref{eq:Vitse1}). However, in \cite{Vitse2005b} the bound of the calculus depends on a factor $C(T)^{5}$, whereas in our Theorem \ref{thm2}, this gets improved to a behavior of  $C(T)(\log C(T) +1)$. Moreover, Corollary \ref{cor1} shows that the same dependence holds true for the power-bound of a Tadmor--Ritt operator. This confirms the result by Bakaev \cite{Bakaev03}, which seems not so well-known, and improves the better known quadratic dependence $C(T)^{2}$, see \cite{Ransford02}, \cite{Vitse2005b}.
\item It is a natural question to ask if the $\|\cdot\|_{\infty,\D}$-norm in Theorem \ref{thm2} can be replaced by the sharper $\|\cdot\|_{\infty,\mathcal{B}_{\eta}}$-norm for some $\eta<\frac{\pi}{2}$. Indeed, Lemma \ref{lem:Bernstein} allows us to do this, see also \eqref{eq:thm213}. However, this leads to an additional factor $(\sin\eta)^{-n}$, which therefore destroys the logarithmic behavior in $\frac{n+1}{m+1}$.\newline
Let us further remark that a polynomially bounded  Tadmor--Ritt operator $T$ (see \eqref{eq:Polybdd}) on a Hilbert space implies an estimate of the form
\begin{equation*}
		\|p(T)\|\lesssim \|p\|_{\infty,\mathcal{B}_{\eta}},
\end{equation*}
for some $\eta<\frac{\pi}{2}$. In other words, $T$ \textit{allows for a bounded $H^{\infty}(\mathcal{B}_{\eta})$-calculus}. However, this is not true for general Banach spaces, see \cite{LancienLeMerdy2013}. More generally, including the Hilbert space case, if one assumes that $T$ is $R$-Ritt (see Section \ref{subsec:GeneralSQFE}), then polynomial-boundedness does indeed imply a bounded $H^{\infty}(\mathcal{B}_{\eta})$-calculus, see \cite[Proposition 7.6]{LeMerdy2014a} on arbitrary Banach spaces.
\end{enumerate}
\end{remark}
%\begin{problem}
%s the dependence $C(T)(\log C(T) +1)$ sharp in the power-bound of a Tadmor--Ritt operator? Or is it actually a \textit{linear} dependence?
%\end{problem}
%\subsection{Stability of difference schemes and Palencia's result}
%It is well known that the Cayley transform
%\begin{equation}
%	z\mapsto Cay(z) = \frac{1-z}{z+1}
%\end{equation}
%maps the open right-half plane to the unit disc.
%\begin{lemma}
%	Let $r>0$ and $\mathcal{F}\subset H(\D\cup B_{r_{0}}(1))$. Assume that
%	\begin{align}
%	|f(z)| {}&\leq 1,\ z\in\D, f\in\mathcal{F},\\
%	s_{\mathcal{F}}:={}&\sup\left\{|f(z)|:f\in\mathcal{F},z\in\D\cup B_{r_{0}}(1)\right\}<\infty,
%	\end{align}
%	and that there exists $z_{0}\in\D$ such that
%	\begin{equation}
%		 \gamma:=\sup\left\{|f(z_{0})|:f\in\mathcal{F}\right\}<1.
%	\end{equation}
%	Then, for any $g=\prod_{j=1}^{m} f_{j}$ with $f_{j}\in\mathcal{F}$, $r\in(0,\frac{r_{0}}{2})$, $\eta\in(0,\frac{\pi}{2})$,
%	\begin{equation}\label{eq:lem2}
%			 \tilde{G}(m,\eta,r) := \int_{\partial\Omega_{\eta,r}} \frac{|g(z)|}{ |z-1|} \ |dz| \leq {\rm const}_{r_{0},s_{\mathcal{F}},\gamma}\cdot \mathcal{C}(r,m,\eta),
%			\end{equation}
%	with $\partial \Omega_{\eta,r}$ and $\mathcal{C}(r,m,\eta)$ as in Lemma \ref{lem2}.
%	\end{lemma}
%

\section{The effect of discrete square function estimates - Hilbert space}\label{subsec:HSSQFE}
In the following we will show that discrete square function estimates improve the dependence in the way that $\log\frac{n+1}{m+1}$ in (\ref{eq:thm2}) gets replaced by its square root.

\begin{definition}[Hilbert space square function estimate]\label{Def:SQFE}
Let $T$ be a bounded operator on a Hilbert space $X$. We say that \textit{$T$ satisfies square function estimates} if there exists a $K>0$ such that 
\begin{equation}\label{eq:SQFE}
	\|x\|_{T}^{2}:=\sum_{k=1}^{\infty}k\|T^{k}x-T^{k-1}x\|^{2} \leq K^{2} \|x\|^{2},\quad \forall x\in X.
\end{equation} 
\end{definition}
Square function estimates are a well-known tool characterizing bounded $H^{\infty}$-calculi for sectorial operators, going back to McIntosh's seminal work in the 80ties \cite{mcintoshHinf}. From the 90ties on, $H^{\infty}$-calculus has proved very useful in the study of maximal regularity. In \cite{cowlingdoustmcintoshyagi} a suitable $L^{p}$-version of square function estimates was introduced which  then got further adapted to general Banach spaces by Kalton and Weis in the unpublished note \cite{KaltonWeisUnpublished}, see also \cite{KunstmannWeis04} and the references therein.
Maximal regularity for discrete-time difference equations were investigated in \cite{Blunck1,Blunck2}.
Discrete square function estimates for Tadmor--Ritt operators were studied in \cite{KaltonPortal}. We mention that in the literature there exists a whole scale of square functions, see \cite[Section 3]{LeMerdy2014a}, whereas we only use the specific form in Definition \ref{Def:SQFE}. 

As for sectorial operators, for non-Hilbert (typically, $L^{p}$-) spaces suitable square function estimates have to be redefined for Tadmor--Ritt operators using Rademacher means.
For the moment we will restrict ourselves to the Hilbert space case and leave the general Banach space case for Section \ref{subsec:GeneralSQFE}.\newline
The following characterization of bounded $H^{\infty}$ calculus for Tadmor--Ritt operators was recently proved in \cite{LeMerdy2014a}. For the rest of the section we want to emphasize that on Hilbert spaces the notions of $R$-Ritt and Tadmor operator coincide, whereas on general Banach spaces $R$-Ritt is stronger than Tadmor--Ritt. For a definition of $R$-Ritt operators and square function estimates on general Banach spaces, we refer to Section \ref{subsec:GeneralSQFE}.
\begin{theorem}[Le Merdy 2014, {\cite[Corollary 7.5]{LeMerdy2014a}}]\label{thm:bddcalc}
Let $T$ be a Tadmor--Ritt operator on a Banach space $X$. Consider the assertions
\begin{enumerate}[label=(\roman*)]
\item\label{thm:bddcalc1} $T$ is $R$-Ritt and both $T$ and $T^{*}$ satisfy square function estimates.
\item\label{thm:bddcalc2} For some $\eta\in(0,\frac{\pi}{2})$,
\begin{equation}\label{eq:bddcalc}
	\|f(T)\|\lesssim \|f\|_{\infty,\mathcal{B}_{\eta}}\quad \forall f\in H_{0}^{\infty}(\mathcal{B}_{\eta}),
\end{equation}
where $H_{0}^{\infty}(\mathcal{B}_{\eta})$ is defined in Remark \ref{rem21}.
\end{enumerate}
Then, \ref{thm:bddcalc1} $\Rightarrow$ \ref{thm:bddcalc2}. \newline If $X$ is UMD space (in particular, a Hilbert space), then \ref{thm:bddcalc2} $\Rightarrow$ \ref{thm:bddcalc1}.
%$T$ has a bounded $H^{\infty}(\mathcal{B}_{\eta})$ calculus for
\end{theorem}
The assumption on (geometry of) the Banach space for the direction \ref{thm:bddcalc1} to \ref{thm:bddcalc2} can be further generalized to  $X$ \textit{having property $(\Delta)$}, see \cite{LeMerdy2014a,KaltonWeis01}.
%A converse of the theorem holds if we assume that the considered Banach space has property $(\Delta)$ (thus, in particular for Hilbert spaces and UMD spaces), see \cite{LeMerdy2014a,KaltonWeis01}. Then, if $T$ is Tadmor--Ritt of type $\alpha\in(0,\frac{\pi}{2})$ and (\ref{eq:bddcalc}) holds for some $\eta\in(\alpha,\frac{\pi}{2})$, it follows that $T$ is $R$-Ritt and that both $T$ and $T^{*}$ satisfy square function estimates, \cite[Corollary 7.5]{LeMerdy2014a}.\newline
In \cite[Proposition 8.1]{LeMerdy2014a} it is further shown that there exist Tadmor--Ritt operators (even on Hilbert spaces) such that (only) $T$ satisfies square function estimates, but (\ref{eq:bddcalc}) does not hold.
However, we will see that having square function estimates for $T$ (or $T^{*}$) does improve the functional calculus estimate in Theorem \ref{thm2}.
Note that for a Tadmor--Ritt operator $T$ of type $\theta$ and $r\in(0,1)$, $rT$ is again Tadmor--Ritt with
\begin{equation*}
C(rT)=\sup_{|\lambda|>1}\left\|(\lambda-1)\tfrac{1}{r}\left(\tfrac{\lambda}{r}-T\right)^{-1}\right\|\leq C(T) \sup_{|\lambda|>1}\left|\frac{\lambda-1}{\lambda-r}\right|=\frac{2C(T)}{1+r}.
\end{equation*}
We remark that moreover $\lim_{r\nearrow1}f(rT)=f(T)$ for $f\in H_{0}^{\infty}(\mathcal{B}_{\eta})$ with $\eta\in(\theta,\frac{\pi}{2})$, see \cite[Lemma 2.3]{LeMerdy2014a}.
%By the maximum principle, the supremum is attained on the unit circle. An easy calculation shows that it can only be attained at 1 or -1, hence -1 gives the constant.
\begin{lemma}\label{lem:SFQE}
Let $T$ be a Tadmor--Ritt operator on a Hilbert space $X$. For $m\in\N\cup\left\{0\right\}$, $r\in(0,1)$,
\begin{equation}\label{eq:rTTR}
	\|(rT)^{m}x\|_{rT} \leq a r^{m} \sqrt{b+\log\left(1-\frac{1}{2(m+1)\log r}\right)} \ \|x\|\quad \forall x\in X,
\end{equation}
with $a=\sqrt{2}c_{1,T}$ and $b=1+\frac{Pb(T)^{2}}{c_{1,T}^{2}}$, where $c_{1,T}$ and $Pb(T)$ are defined in \eqref{eq1:lemcharRitt}.
\end{lemma}
\begin{proof}
Clearly, $rT$ is a Tadmor--Ritt operator. % with $C(rT)\leq \frac{2}{1+r}C(T)$, see \eqref{eq:rTTR}.
By definition,
\begin{align}
\|(rT)^{m}x\|_{rT}^{2}={}&r^{2m}\sum_{k=1}^{\infty} k\|r^{k}T^{k+m}x-r^{k-1}T^{k-1+m}x\|^{2}\notag\\
		\leq{}&r^{2m}\sum_{k=1}^{\infty}kr^{2(k-1)}\left(2\|T^{k+m}x-T^{k-1+m}x\|^{2}+2\|(1-r)T^{k+m}x\|^{2}\right)\notag\\
		\leq{}&r^{2m}\sum_{k=1}^{\infty}kr^{2(k-1)}\left(\frac{2c_{1,T}^{2}}{(k+m)^{2}}+2(1-r)^{2}Pb(T)^{2}\right)\|x\|^{2},\label{eq2:lemSFQE}
\end{align}
where $c_{1,T}=\sup_{n\in\N}\|n(T^{n}-T^{n-1})\|$ which is finite by Lemma \ref{lem:charRitt}. Since $\frac{k}{(k+m)^{2}}\leq\frac{1}{k+m}$,
\begin{align}
		\sum_{k=1}^{\infty}\frac{kr^{2(k-1)}}{(k+m)^{2}}\leq{}&\sum_{k=0}^{\infty}\frac{r^{2k}}{k+1+m}\notag\\
					\leq {}& \frac{1}{m+1}+\int_{0}^{\infty}\frac{e^{2x\log r}}{x+1+m}\ dx\notag\\
					={}&\frac{1}{m+1}+ r^{-2(m+1)}{\rm Ei}(-2(m+1)\log r)\notag\\
					\leq{}&\frac{1}{m+1}+\log\left(1-\frac{1}{2(m+1)\log r}\right),\label{eq3:lemSFQE}
\end{align}
where the last step follows by \eqref{est:Ei}. Using this and the fact that $\sum_{k=1}^{\infty}kr^{2(k-1)}=\frac{1}{(1-r^{2})^{2}}$, we can conclude in \eqref{eq2:lemSFQE} that
\begin{align*}
		\left\|(rT)^{m}x\right\|_{rT}^{2}\leq{}&r^{2m}\left[2c_{1,T}^{2}\left(\frac{1}{m+1}+\log\left(1-\frac{1}{2(m+1)\log r}\right)\right)+\frac{2Pb(T)^{2}}{(1+r)^{2}}\right]\|x\|^{2}\\
			\leq{}&2c_{1,T}^{2}r^{2m}\left(b+\log\left(1-\frac{1}{2(m+1)\log r}\right)\right)\|x\|^{2},
\end{align*}
for $b=1+\frac{Pb(T)^{2}}{c_{1,T}^{2}}$. 
\end{proof}
Another lemma, we will need, is the following result relating square function estimates for $T$ and $rT$ as $r\nearrow1$. This can be seen as a discrete analog of \cite[Proposition 3.4]{LeMerdy2003}.
\begin{lemma}\label{lem:expbdd}
Let $T$ be a Tadmor--Ritt operator on a Hilbert space. Then, the following are equivalent
\begin{enumerate}[label={{(\roman*)}}]
\item\label{lem:expbddit1} $T$ satisfies square function estimates.
\item\label{lem:expbddit2} $rT$ satisfies square function estimates uniform in $r\in(0,1)$, i.e., 
\begin{equation*}
     \exists K>0\ \forall r\in(0,1)\forall x\in X:\quad		\|x\|_{rT} \leq K\ \|x\|.
\end{equation*}
\end{enumerate}
\end{lemma}
\begin{proof}
This follows from the more general Lemma \ref{lem:expbdd2} in Section \ref{subsec:GeneralSQFE}.
%The proof is similar to one for the continuous time analog \cite[Proposition 3.4]{LeMerdy2003} and is based on using the identity
%\begin{equation}\label{eq:idexpbdd}
%(I-T)T^{k}x=(I-rT)T^{k}x+(1-r)T^{k+1}x.
%\end{equation}
%This implies that for $\|x\|=1$,
%\begin{equation*}
%	\sum_{k=1}^{\infty}kr^{2k}\|(I-T)T^{k}x\|^{2}\leq \sum_{k=1}^{\infty}k\|(I-rT)(rT)^{k}x\|^{2} + Pb(T)^{2}(1-r)^{2}\sum_{k=1}^{\infty}kr^{2k}.
%\end{equation*}
%It is easy to see that the second term is bounded in $r\in(0,1)$, hence, by Fatou's lemma, we get that \ref{lem:expbddit2} implies \ref{lem:expbddit1}.
%The other direction also follows, with a similar estimation, from \eqref{eq:idexpbdd}.
\end{proof}
The following theorem is essentially Le Merdy's key argument to prove that \ref{thm:bddcalc1} implies \ref{thm:bddcalc2} in Theorem \ref{thm:bddcalc}. 
As we need its precise form, we state it explicitly. For a proof we refer to \cite[Proof of Theorem 7.3]{LeMerdy2014a}. For a definition of \textit{$R$-Ritt operator of $R$-type $\theta$} we refer to Section \ref{subsec:GeneralSQFE}. For the moment it suffices to remark that on Hilbert spaces this notion is equivalent to the of one a Tadmor--Ritt operator of type $\theta$, see Section \ref{subsec:GeneralSQFE}.
\begin{theorem}[Le Merdy 2014]\label{thm:LeMerdy}
Let $T$ be a $R$-Ritt operator of $R$-type $\theta$ on a Banach space $X$. Let $0<\theta<\eta<\frac{\pi}{2}$. Then, there exists $c=c(\eta,C(T))>0$ such that
\begin{equation*}
|\langle y, p(T)x\rangle| \leq c \cdot \|p\|_{\infty,\mathcal{B}_{\eta}}\cdot \|x\|_{T}\cdot \|y\|_{T^{*}},
\end{equation*}
for any polynomial $p$, $x\in X$ and $y\in X^{*}$.\newline
(Note that the right-hand-side is allowed to be $\infty$).
\end{theorem}
Combining Theorem \ref{thm:LeMerdy} and Lemma \ref{lem:SFQE} yields the following refinement of Theorem \ref{thm2}.
\begin{theorem}\label{thm3}
Let $T$ be a Tadmor--Ritt operator on a Hilbert space $X$. Assume that either $T$ or $T^{*}$ satisfies square function estimates. Then, for integers $0\leq m\leq n$ and $p(z)=\sum_{j=m}^{n}a_{j}z^{j}$,
\begin{equation*}
\|p(T)\| \leq acKe^{\frac{1}{2}}\cdot  \sqrt{b+\log\frac{n+2}{m+1}}\cdot \|p\|_{\infty,\D},
\end{equation*}
with $K,a,b,c$ defined in \eqref{eq:SQFE}, Lemma \ref{lem:SFQE} and Theorem \ref{thm:LeMerdy}, respectively.
\end{theorem}
\begin{proof}
Since $X$ is a Hilbert space, $T$ is $R$-Ritt of type $\theta=\arccos\frac{1}{C(T)}$. Let $r\in(0,1)$ and choose $\eta\in(\theta,\frac{\pi}{2})$. Define $p_{\frac{1}{r}}(z)=p(\frac{z}{r})$. It is easy to see that $p_{\frac{1}{r}}(rT)=p(T)$ since $p$ is a polynomial.
Furthermore, we write $p(z)=z^{m}q(z)$ for $q$ having degree $n-m$. Therefore, for all $x\in X$,
\begin{equation}\label{eq:thm3eq2}
	p(T)x=q_{\frac{1}{r}}(rT)(rT)^{m}x.
\end{equation}
W.l.o.g.\ let $T^{*}$ satisfy square function estimates. Hence, by Lemma \ref{lem:expbdd}, $\|y\|_{rT^{*}}\leq K \|y\|$ for all $y\in X^{*}$ and all $r\in(0,1)$.
Applying Theorem \ref{thm:LeMerdy} for $rT$ and $p=q_{\frac{1}{r}}$ yields
\begin{equation}\label{eq:thm3eq1}
	|\langle y, q_{\frac{1}{r}}(rT)(rT)^{m}x\rangle | \leq cK\cdot \|q_{\frac{1}{r}}\|_{\infty,\D}\cdot \|(rT)^{m}x\|_{rT}\cdot \|y\|,
\end{equation}
for $x\in X, y\in X^{*}$ where we used that $\mathcal{B}_{\eta}\subset \D$. By Lemma \ref{it1:Bernstein} and the maximum principle,  $\|q_{\frac{1}{r}}\|_{\infty,\D}\leq r^{m-n}\|q\|_{\infty,\D}=r^{m-n}\|p\|_{\infty,\D}$. Therefore, and by Lemma \ref{lem:SFQE}, Eq.\ \eqref{eq:thm3eq1} yields
\begin{equation*}
\|q_{\frac{1}{r}}(rT)(rT)^{m}\| \leq acK \sqrt{b+\log\left(1-\tfrac{1}{2(m+1)\log r}\right)}\cdot r^{2m-n}\cdot \|p\|_{\infty,\D}.
\end{equation*}
Choose $r=e^{-\frac{1}{2(n-m+1)}}$. Then $1-\frac{1}{2(m+1)\log r}=\frac{n+2}{m+1}$ and $r^{2m-n}=e^{\frac{n-2m}{2(n-m+1)}}<e^{\frac{1}{2}}$. Thus, by (\ref{eq:thm3eq2}),
\begin{equation*}
 \|p(T)\| \leq acKe^{\frac{1}{2}}\sqrt{b+\log\frac{n+2}{m+1}}\cdot \|p\|_{\infty,\D}.
\end{equation*}
\end{proof}
\begin{remark}\hfill
\begin{enumerate}
\item The proof idea of Theorem \ref{thm3} can also be used for an alternative proof of the logarithmic behavior in Theorem \ref{thm2}, if we do a similar computation for $\|T^{m}y\|_{rT^{*}}$  (instead of assuming square function estimates $\|y\|_{T}\lesssim \|y\|$). This finally yields another factor of the form $\sqrt{\tilde{b}+\log\frac{n+2}{m+1}}$.
\item As explained in Remark \ref{rem1}, in Theorem \ref{thm3} we can also derive `sharper' estimates in the $\|\cdot\|_{\infty,\mathcal{B}_{\eta}}$-norm at the price that additional factors of the form $(\sin\eta)^{-n}$ enter the estimate.
\end{enumerate}
\end{remark}
\section{Discrete square function estimates on general Banach spaces}\label{subsec:GeneralSQFE}
As indicated in Section \ref{subsec:HSSQFE}, for non-Hilbert spaces, Definition \ref{Def:SQFE} is not suitable for characterizing boundedness of the $H^{\infty}$-calculus. For $L^{p}$-spaces the proper replacement is given by 
\begin{equation}\label{eq:LpSQFE}
	\|x\|_{T}:= \left\|\left(\sum\nolimits_{k=1}^{\infty}k|T^{k}x-T^{k-1}x|^{2}\right)^{\frac{1}{2}}\right\|_{L^{p}} \lesssim \|x\|
\end{equation}
where $T$ is a Tadmor--Ritt operator on $L^{p}(\Omega)$, $p\in[1,\infty)$, for some measure space $(\Omega,\mu)$, see \cite{KaltonPortal}, \cite{LeMerdy2014a} and the references therein. By Fubini's theorem, this definition coincides with Definition \ref{Def:SQFE} if $p=2$.\newline
However, to cover general Banach spaces, we need the following generalization using Rademacher averages. This approach (for sectorial operators), paving the way for a lot of research in this field, was introduced by Kalton and Weis in their `famous' unpublished note, see the new preprint \cite{KaltonWeisUnpublished}. For an excellent overview on the topic we refer to \cite{KunstmannWeis04}.
The discrete version of these general square function estimates for Tadmor--Ritt operators recently appeared in \cite{LeMerdy2014a}. 

We briefly recap the definition of the needed Rademacher norms. For more details, we refer to \cite{LeMerdy2014a,KunstmannWeis04}.
For $k\geq1$, we define the Rademacher function $\veps_{k}(t)={\rm sgn}(\sin(2^{k}\pi t))$. It is easy to see that $(\veps_{k})_{k\geq1}$ forms an orthonormal basis in $L^{2}(I)$ with $I=[0,1]$. For a Banach space $X$ let us consider the linear span of  elements $\veps_{k}\otimes x =(t\mapsto\veps_{k}(t) x)$, $k\geq0$, $x\in X$, in the Bochner space $L^{2}(I,X)$. Denote the closure of this set, w.r.t.\ the norm in $L^{2}(I,X)$, by ${\rm Rad}(X)$. Hence, ${\rm Rad}(X)$ becomes a Banach space with the norm
\begin{equation*}
	\|\tilde{x}\|_{{\rm Rad(X)}}= \left(\int_{I}^{} \left\|\sum\nolimits_{k} \veps_{k}(t)x_{k} \right\|^{2}  dt\right)^{\frac{1}{2}},
\end{equation*} 
for elements $\tilde{x}=\sum_{k}\veps_{k}\otimes x_{k}$ with $(x_{k})_{k}$ being a finite family in $X$. By orthonormality of the Rademacher functions it follows that 
\begin{equation}\label{eq:Rad}
	{\rm Rad}(X) = \left\{\sum\nolimits_{k=1}^{\infty} \veps_{k} \otimes x_{k}:x_{k}\in X, \text{the sum converges in }L^{2}(I,X)\right\}.
\end{equation}
Now we can define a general square function by
\begin{equation*}
\|x\|_{T} =  \left\|\sum\nolimits_{k=1}^{\infty} \veps_{k} \otimes k(T^{k}x-T^{k-1}x)\right\|_{{\rm Rad}(X)},
\end{equation*}
where we set $\|x\|_{T}=\infty$ if $\sum_{k}\veps_{k}\otimes k(T^{k}x-T^{k-1}x)\notin{\rm Rad}(X)$.
\begin{definition}[Square function estimates for Tadmor--Ritt operators]\label{def:abstractSQFE}
Let $T$ be a Tadmor--Ritt operator on a Banach space $X$. We say that $T$ satisfies \textit{(abstract) square function estimates}, if
there exists $K_{T}>0$ such that for all $x\in X$,
\begin{equation}\label{eq:SQFE2}
\|x\|_{T} =  \left\|\sum\nolimits_{k=1}^{\infty} \veps_{k} \otimes k^{\frac{1}{2}}(T^{k}x-T^{k-1}x)\right\|_{{\rm Rad}(X)}\leq K_{T} \|x\|.
\end{equation}
\end{definition}
Note that if $X$ is a Hilbert space, as a consequence of Parseval's identity, this definition of square function estimates coincides with the one given in Definition \ref{Def:SQFE}. Precisely, for any finite sequence $(x_{k})_{k}\in X$,
\begin{equation}\label{eq:RademacherParseval}
	\left\|\sum\nolimits_{k}\veps_{k} \otimes x_{k}\right\|_{{\rm Rad}(X)}=(\sum\nolimits_{k}\|x_{k}\|^{2})^{\frac{1}{2}},
\end{equation}
which shows that both definitions of square functions estimates coincide. Further, it can be shown that for $X=L^{p}=L^{p}(\Omega,\mu)$ ($p\in[1,\infty)$ and $(\Omega,\mu)$ being $\sigma$-additive),  
\begin{equation*}
	\left\|\sum\nolimits_{k}\veps_{k} \otimes x_{k}\right\|_{{\rm Rad}(L^{p})}\sim \left\|\left(\sum\nolimits_{k}|x|^{2}\right)^{\frac{1}{2}}\right\|_{L^{p}},
\end{equation*}
see \cite[Remark 2.9]{KunstmannWeis04}.
 Hence, \eqref{eq:LpSQFE} is equivalent to having square function estimates using Rademacher averages.
 
The notion of $R$-boundedness emerges naturally in the framework of the space ${\rm Rad}(X)$.
After being introduced in \cite{BerksonGillespie1994}, it has been proved very useful in the study of maximal regularity, see \cite{KunstmannWeis04} for a detailed introduction.

\begin{definition}\label{def:Rboundedness}
Let $X$ be a Banach space and $\mathcal{T}\subset\Bo(X)$ a set of bounded operators. Then, $\mathcal{T}$ is called \textit{$R$-bounded} if there exists a constant $M$  such that for any finite family $(T_{k})){k}\in\mathcal{T}$, and finite sequence $(x_{k})_{k}\subset X$,
\begin{equation}\label{eq:Rbounded}
	\left\|\sum\nolimits_{k} \veps_{k}\otimes T_{k}x_{k}\right\|_{{\rm Rad}(X)} \leq M \left\| \sum\nolimits_{k}\veps_{k}\otimes x_{k} \right\|_{{\rm Rad}(X)}.
\end{equation}
The smallest possible constant $C$ is called the $R$-bound.
\end{definition}
By \eqref{eq:RademacherParseval}, it follows that for Hilbert spaces the notion of $R$-boundedness of $\mathcal{T}$ coincides with (uniform) boundedness of $\mathcal{T}$ in the operator norm. However, in general, $R$-boundedness only implies boundedness, see \cite{ArendtBu2002}.\newline
Now we are able to introduce $R$-Ritt operators, which first appeared in \cite{Blunck1,Blunck2}. Nonetheless the notion \textit{$R$-Tadmor--Ritt} would be more consistent in this Chapter, we use the name \textit{$R$-Ritt} following Le Merdy \cite{LeMerdy2014a}.  For Hilbert spaces, the following notion is equivalent to the one of a Tadmor--Ritt operator, see Lemma \ref{lem:charRitt}.
\begin{definition}\label{def:RRitt}
An operator $T$ on a Banach space $X$ is called \textit{$R$-Ritt} if the sets
\begin{equation*}
\left\{T^{n}:n\in\N\right\} \ \text{ and }\ \left\{n(T^{n}-T^{n-1}):n\in\N\right\}
\end{equation*}
are $R$-bounded. We denote the bounds by $Pb^{R}(T)$ and $c_{1,T}^{R}$, respectively.
\end{definition}
%By Lemma \ref{lem:charRitt}, if $X$ is a Hilbert space, this definition is equivalent to the one of a Tadmor--Ritt operator, and for general $X$, the set of $R$-Ritt operators is only contained in $TR(X)$.
By Lemma \ref{lem:charRitt}, an $R$-Ritt operator is always a Tadmor--Ritt operator and the notions coincide  on Hilbert spaces. Moreover, the following $R$-Ritt version of Lemmata \ref{Lemma1} and \ref{lem:charRitt} holds. For a proof, see \cite[Lemma 5.2]{LeMerdy2014a} and \cite{Blunck1}.
\begin{lemma}\label{lem1:RRitt}
Let $T$ be a bounded operator on a Banach space $X$. The following assertions are equivalent.
		\begin{enumerate}[label={(\roman*)}]
		\item\label{Lemma1RRittit1} $T$ is $R$-Ritt.
		\item\label{Lemma1RRittit2} $\sigma(T)\subset \overline{\mathcal{B}_{\theta}}$ for some $\theta\in[0,\frac{\pi}{2})$ and  for all $\eta\in(\theta,\tfrac{\pi}{2}]$
		 \begin{equation} \label{eq:RRitttheta}
				 \left\{(z-1)R(z,T): z\in\C\setminus\overline{\mathcal{B}}_{\eta}\right\} \text{ is $R$-bounded}.
		\end{equation}
	\end{enumerate}
	In this case, we say that $T$ is of \textit{$R$-Ritt type $\theta$}.
\end{lemma}
%operator $T$ it holds that $\sigma(T)\subset \overline{\mathcal{B}_{\theta}}$  for some $\theta\in(0,\frac{\pi}{2})$ and that the set $\left\{(\lambda-T)R(\lambda,T):\lambda\in\C\setminus\overline{\mathcal{B}_{\eta}}\right\}$ is $R$-bounded for all $\eta\in(\theta,\frac{\pi}{2})$, see \cite[Lemma 5.2]{LeMerdy2014a} and \cite{Blunck1}. Therefore, we define a $R$-Ritt operator $T$ to be \textit{of $R$-type $\theta$}, if the above holds.
Now we are ready to prove the corresponding $R$-Ritt version of the results in Section \ref{subsec:HSSQFE} for general Banach spaces.
\begin{lemma}\label{lem:SFQE2}
Let $T$ be a $R$-Ritt operator on a Banach space $X$. For $m\in\N\cup\left\{0\right\}$, $r\in(0,1)$,
\begin{equation*}
	\|(rT)^{m}x\|_{rT} \leq a r^{m} \sqrt{b_{R}+\log\left(1-\frac{1}{2(m+1)\log r}\right)} \ \|x\|\quad \forall x\in X,
\end{equation*}
with $a_{R}=\sqrt{2}c_{1,T}^{R}$ and $b_{R}=1+\frac{Pb^{R}(T)^{2}}{(c_{1,T}^{R})^{2}}$, where $c_{1,T}^{R}, Pb^{R}(T)$ are defined in Def.\ \ref{def:RRitt}.
\end{lemma}
\begin{proof}
The proof technique is very similar to the proof of Lemma \ref{lem:SFQE}. Therefore, we will focus on the arguments involving $R$-boundedness.
Since $rT$ is a Tadmor--Ritt operator, we have, see \ref{eq:SQFE2},
\begin{align}
\|T^{m}x\|_{rT}={}&\left\|\sum_{k=1}^{\infty} \veps_{k}\otimes k^{\frac{1}{2}}(r^{k}T^{k+m}x-r^{k-1}T^{k-1+m}x)\right\|_{{\rm Rad}(X)}\notag\\
		\leq{}&\left\|\sum_{k=1}^{\infty}\veps_{k}\otimes\left[(T^{k+m}-T^{k-1+m})+(1-r)T^{k+m}\right] k^{\frac{1}{2}}r^{k-1}x\right\|_{{\rm Rad}(X)}\notag\\
		\leq{}&c_{1,T}^{R}\left\|\sum_{k=1}^{\infty}\veps_{k}\otimes \frac{k^{\frac{1}{2}}r^{k-1}}{k+m}x\right\|_{{\rm Rad}(X)}+ \notag\\{}&\qquad+(1-r)Pb^{R}(T) \left\|\sum_{k=1}^{\infty}\veps_{k}\otimes k^{\frac{1}{2}}r^{k-1}x\right\|_{{\rm Rad}(X)} ,\label{eq2:lemSFQE2}
\end{align}
where the last step follows since $T$ is $R$-Ritt. By the definition of the ${\rm Rad}(X)$-norm, and Parseval's identity (for $L^{2}[0,1]$), the first norm in \eqref{eq2:lemSFQE2} equals
\begin{align*}
	\left\|\sum_{k=1}^{\infty}\veps_{k}\otimes \frac{k^{\frac{1}{2}}r^{k-1}}{k+m}x\right\|_{{\rm Rad}(X)}^{2} ={}& \int_{0}^{1} \left\|\sum_{k=1}^{\infty} \veps_{k}(t)\tfrac{k^{\frac{1}{2}}r^{k-1}}{k+m}x\right\|_{X}^{2}dt\\
	 ={}&\|x\|^{2}\int_{0}^{1}\left|\sum_{k=1}^{\infty}\veps_{k}(t)\tfrac{k^{\frac{1}{2}}r^{k-1}}{k+m}\right|^{2} dt\\
	 ={}&\|x\|^{2}\sum_{k=1}^{\infty}\left|\frac{k^{\frac{1}{2}}r^{k-1}}{k+m}\right|^{2}.
\end{align*}
The remaining series can be estimated as in the Hilbert space proof. Analogously, the second norm in \eqref{eq2:lemSFQE2} can be computed. Therefore, we derive, 
\begin{align*}
		\|(rT)^{m}x\|_{rT}\leq{}&r^{m}\left[c_{1,T}^{R}\left(\frac{1}{m+1}+\log\left(1-\frac{1}{2(m+1)\log r}\right)\right)^{\frac{1}{2}}+\frac{Pb^{R}(T)}{(1+r)}\right]\|x\|\\
			\leq{}&\sqrt{2}c_{1,T}^{R}r^{2m}\left(b_{R}+\log\left(1-\frac{1}{2(m+1)\log r}\right)\right)^{\frac{1}{2}}\|x\|,
\end{align*}
for $b_{R}=1+\frac{Pb^{R}(T)^{2}}{(c_{1,T}^{R})^{2}}$. 
\end{proof}
We further need the generalization of Lemma \ref{lem:expbdd} to (abstract) square function estimates.
\begin{lemma}\label{lem:expbdd2}
Let $T$ be a $R$-Ritt operator on a Banach space $X$. Then, the following are equivalent.
\begin{enumerate}[label={\textit{(\roman*)}}]
\item\label{lemR:expbddit1} $T$ satisfies (abstract) square function estimates.
\item\label{lemR:expbddit2} $rT$  satisfies (abstract) square function estimates uniform in $r\in(0,1)$,  
\begin{equation*}
     \exists K>0\ \forall r\in(0,1)\forall x\in X:\quad		\|x\|_{rT} \leq K\ \|x\|.
\end{equation*}
\end{enumerate}
\end{lemma}
\begin{proof}
The proof is similar to one for the continuous time analog \cite[Proposition 3.4]{LeMerdy2003} and is based on using the identity
\begin{equation}\label{eq:idexpbdd}
(I-T)T^{k}x=(I-rT)T^{k}x+(1-r)T^{k+1}x.
\end{equation}
This yields, using that $T$ is $R$-Ritt,
\begin{align*}
	\left\|\sum\nolimits_{k=1}^{\infty}\veps_{k}\otimes k^{\frac{1}{2}}r^{k}(I-T)T^{k}x\right\|_{{\rm Rad}(X)}\leq{}& \left\|\sum\nolimits_{k=1}^{\infty}\veps_{k}\otimes k^{\frac{1}{2}}(I-rT)(rT)^{k}x\right\|_{{\rm Rad}(X)} +\notag\\{}&\ +Pb^{R}(T)(1-r)\left\|\sum\nolimits_{k=1}^{\infty}\veps_{k}\otimes k^{\frac{1}{2}}r^{k}x\right\|_{{\rm Rad}(X)}.%\label{eq2:lemexpbdd2}
\end{align*}
It is easy to see that the second term on the right-hand is bounded in $r\in(0,1)$, because 
the ${\rm Rad}(X)$-norm  equals $(\sum_{k=1}^{\infty}kr^{2k})^{\frac{1}{2}}\|x\|=\|x\|(1-r^{2})^{-1}$ by Parseval's identity. Hence, by Fatou's lemma, we get that \ref{lemR:expbddit2} implies \ref{lemR:expbddit1}.
The other direction also follows, with a similar estimation, from \eqref{eq:idexpbdd}.
\end{proof}
The Banach space version of Theorem \ref{thm3} now follows completely analogously to the Hilbert space proof with Lemmata \ref{lem:SFQE2} and \ref{lem:expbdd2} (instead of Lemmata \ref{lem:SFQE} and \ref{lem:expbdd}).
\begin{theorem}\label{thm32}
Let $T$ be a $R$-Ritt operator on a Banach space $X$. Assume that either $T$ or $T^{*}$ satisfies (abstract) square function estimates. Then, for integers $0\leq m\leq n$ and $p(z)=\sum_{j=m}^{n}a_{j}z^{j}$,
\begin{equation*}
\|p(T)\| \leq a_{R}cK_{T}e^{\frac{1}{2}}\cdot  \sqrt{b_{R}+\log\frac{n+2}{m+1}}\cdot \|p\|_{\infty,\D},
\end{equation*}
with $K_{T},a_{R},b_{R}$ and $c$ defined in \eqref{eq:SQFE2}, Lemma \ref{lem:SFQE2} and Theorem \ref{thm:LeMerdy}, respectively.
\end{theorem}
\section{Sharpness of the estimates}\label{subsec:SharpnessofEstimates}
It is natural to ask whether the deduced functional calculus estimates from Theorems \ref{thm2} and \ref{thm3},
\begin{equation}\label{eq1:sharpness}
	\|p(T)\| \leq aC(T)\left(\log C(T)+b+\log\frac{n+1}{m+1}\right)\|p\|_{\infty,\D},
\end{equation}
and
\begin{equation}\label{eq2:sharpness}
\|p(T)\| \leq a_{2}cK_{T}e^{\frac{1}{2}}\cdot  \sqrt{b_{2}+\log\frac{n+2}{m+1}}\cdot \|p\|_{\infty,\D},
\end{equation}
  for $p\in H^{\infty}[m,n]$, that is $p(z)=\sum_{k=m}^{n}a_{k}z^{k}$, are sharp. Clearly, here `sharpness'  has different aspects depending on the variables $C(T),m,n$ it is referring to. 
For a clear discussion, we distinguish between the following questions.
\begin{enumerate}[label=(\Alph*), itemsep=5pt,leftmargin=30pt]
	\item\label{it:sharpness1} Is \eqref{eq1:sharpness} sharp in the variables $m,n$, with $0\leq m\leq n$?
	\item\label{it:sharpness2} Is \eqref{eq1:sharpness} sharp in the variable $C(T)$ for (some) fixed $m,n$?
	\item \label{itsharpness3} Question \ref{it:sharpness1} for \eqref{eq2:sharpness}.
	\item  \label{itsharpness4} Question \ref{it:sharpness2} for \eqref{eq2:sharpness}.
\end{enumerate}
To answer these questions, we introduce the quantity
\begin{equation}\label{eq:C(T,m,n)}
 C(T,m,n)=\sup\left\{\|p(T)\|: p\in H^{\infty}[m,n], \|p\|_{\infty,\D}\leq 1\right\}.
\end{equation}
Question \ref{it:sharpness1} was discussed Vitse in \cite[Remark 2.6]{Vitse2005b} using the prior works \cite{Vitse04Turndown,Vitse05}.
In particular, she showed that if $X$ contains a complemented isomorphic copy of $\ell^{1}$ or $\ell^{\infty}$ (e.g., infinite-dimensional $L^{1}$ or $C(K)$ spaces), then there exists a Tadmor--Ritt operator on $X$ such that 
\begin{equation*}
	C(T,m,n) \gtrsim \log\frac{ne}{m},		
\end{equation*}
where the involved constant only depends on $X$ and is thereby linked with constant $C(T)$. However, the precise dependence on $C(T)$ is not apparent there. 
If $X$ is an (infinite-dimensional) Hilbert space (more, generally if the Banach space $X$ contains a complemented isomorphic copy of $\ell^{2}$), then for any $\delta\in(0,1)$,
there exists a Tadmor--Ritt operator such that 
\begin{equation*}
	C(T,m,n)\gtrsim \left(\log\frac{ne}{m}\right)^{\delta}.
\end{equation*}
These statements can be generalized to more general spaces $X$ that \textit{uniformly contain uniform copies} of $\ell_{n}^{1}$ (or $\ell_{n}^{2}$ respectively). We refer to \cite{Vitse04,Vitse2005b} for details. \medskip\\
Question \ref{it:sharpness2} can be split up in several cases. If $m=0$, hence $p$ is an arbitrary polynomial of degree $n$, \eqref{eq1:sharpness} implies that $C(T,0,n)\lesssim C(T)(\log C(T)+\log(n+1))$. 
Hence, we observe `linear' asymptotic behavior in $C(T)$ as $n\to\infty$.
In fact, in \cite[Theorem 2.1]{Vitse04} it is shown that it is indeed linear, namely 
\begin{equation*}
	C(T,0,n) \leq (C(T)+1)\log(e^{2}n),
\end{equation*}
and there exists a $T$ on some Banach space $X$ such that $C(T,0,n)\sim \log(e^{2}n)$. We point out that the proof technique, \cite[Theorem 2.1]{Vitse04}, requires $m=0$. \medskip\\\\
However, for $m=n$, Question \ref{it:sharpness2} reduces to the prominent question of the optimal power-bound for $T$. As mentioned in Corollary \ref{cor1}, \eqref{eq1:sharpness} yields 
\begin{equation*}
	C(T,n,n)=\|T^{n}\|\lesssim C(T)(\log C(T) +1),
\end{equation*}
for all $n$. This is so-far the best known power-bound for Tadmor--Ritt operators, see also \cite{Bakaev03}. It remains open whether this can be replaced by a linear $C(T)$-dependence.
Furthermore, motivated by the Kreiss Matrix Theorem \eqref{eq:SpijkerKMT}, it is not clear whether for $N$-dimensional spaces $X$, an estimate of the form
\begin{equation}\label{eq:FiniteTRclaim}
	Pb(T) \leq C(T) g(N)
\end{equation}
for some scalar function $g$ can be achieved, where $g(N)\in o(N)$. Note that the estimate for $g(N)=eN$ trivially holds by \eqref{eq:SpijkerKMT} and the fact that $C_{Kreiss}(T)\leq C(T)$.
\medskip\\
Let us turn to Question \ref{itsharpness3} now.  We want to show sharpness of 
\begin{equation*}
C(T,m,n)\lesssim \sqrt{\log\frac{n+1}{m+1}}
\end{equation*}
under the assumption that $T$ satisfies square function estimates. Therefore, we  construct $T$ as a Schauder basis multiplier, which is a well-known technique to construct unbounded calculi, see e.g., \cite[Chapter 9]{haasesectorial} and \cite{BaillonClement91}, where it was introduced. Let $X$ be a separable infinite-dimensional Hilbert space with a bounded Schauder basis $\left\{\psi_{k}\right\}$. For a sequence $(\lambda_{n})\subset[0,1]$, define the bounded operator $T=\mathcal{M}_{\lambda}$ by 
\begin{equation*}
Tx=\left(\sum\nolimits_{k}x_{k}\psi_{k}\right)=\sum\nolimits_{k}\lambda_{k}x_{k}\psi_{k}, 
\end{equation*}
for finite sequences $(x_{k})\subset \C$. Let $\lambda_{n}=1-2^{-n}$, then $T$ is Tadmor--Ritt, see \cite[Proposition 8.2]{LeMerdy2014a}. 
With this setting we can use the following argument from \cite[Proof of Theorem 2.1]{Vitse05}. Let $\delta\in(0,1)$. 
If for the uniform basis constant $ub(\left\{\psi_{k}\right\}_{k=1}^{N})$ it holds that $ub(\left\{\psi_{k}\right\}_{k=1}^{N})\gtrsim N^{\delta}$, i.e.
\begin{equation}\label{eq:UB}
	\exists c>0~ \forall N\in\N:\quad \sup\left\{\left\|\sum_{k=1}^{N}\alpha_{k}x_{k}\psi_{k}\right\|:|\alpha_{k}|\leq1, \left\|\sum_{k=1}^{N}x_{k}\psi_{k}\right\|\leq1\right\} \geq c N^{\delta},
\end{equation}
then $C(T,m,n)\gtrsim  \left(\log\frac{n+1}{m+1}\right)^{\delta}$.

As for sectorial Schauder multipliers, it holds that $T$ satisfies square function estimates if the basis is \textit{Besselian}, i.e., $ \exists c_{\psi}>0$
\begin{equation}\label{eq:RittBesselian}
	c_{\psi}\left(\sum\nolimits_{k}|x_{k}|^{2}\right)^{\frac{1}{2}} \leq \left\|\sum\nolimits_{k}x_{k}\psi_{k}\right\|,
\end{equation}
for finite sequences $(x_{k})\subset\C$, see \cite[Theorem 5.2]{LeMerdy2003} and \cite[Theorem 8.2]{LeMerdy2014a}. 
Note that \eqref{eq:RittBesselian} already implies that $ub(\left\{\psi_{k}\right\}_{k=1}^{N})\leq c_{\psi}m(\psi) \sqrt{N}$, where $m(\psi)=\sup_{k}\|\psi_{k}\|$.
It remains to find a Besselian basis $\left\{\psi_{k}\right\}$  such that \eqref{eq:UB} is fulfilled for $\delta\in(0,\frac{1}{2})$. 
Indeed, such an example can be constructed for an $L^{2}$-space on the unit circle with suitable weight, see \cite[Thm.~5.2]{LeMerdy2003}, \cite{Schwenninger15a}, and \cite[Section 4.3]{SpijkerTracognaWelfert03}.
In fact, the example in \cite[Thm.~4.5]{Schwenninger15a} gives a basis
\begin{equation*}
\psi_{2k}(t)=|t|^{-\beta}e^{ikt},\quad \psi_{2k+1}(t)=|t|^{-\beta}e^{-ikt},\ k\in\N_{0},
\end{equation*}
 (there, the notation is $\psi^{*}$) with $\beta\in(\frac{1}{3},\frac{1}{2})$. Moreover, it is shown that there exist elements $x,y\in L^{2}$ such that 
\begin{equation*}
|x_{n}|\sim n^{3\beta-1}, \qquad |y_{n}|\sim n^{\beta-1}, \ n\in\N,
\end{equation*}
where $x=\sum_{n}x_{n}\psi_{n}$ and $y\in\sum_{n}y_{n}\psi_{n}^{*}$, and where $\left\{\psi_{n}^{*}\right\}$ denotes the dual basis such that $\langle \psi_{k}^{*},\psi_{n}\rangle =\delta_{nk}$. Choosing $|\alpha_{n}|=1$ such that $\alpha_{n}x_{n}y_{n}\in\R_{\geq0}$, we deduce 
\begin{align*}
	|\langle y, \sum_{n=1}^{N}\alpha_{n}x_{n}\psi_{n}\rangle| = {}&\sum_{n=1}^{N}\alpha_{n}x_{n}y_{n}\gtrsim \sum_{n=1}^{N}n^{3\beta-2}\sim N^{3\beta-1}. 
\end{align*}
Since $\|\sum_{n=1}^{N}x_{n}\psi_{n}\|\leq b(\psi) \|x\|$, \eqref{eq:UB} follows for $\delta=3\beta-1\in(0,\frac{1}{2})$.
Therefore, we have proved the following result, which answers \ref{itsharpness3} for Hilbert spaces.
\begin{theorem}There exists a Hilbert space such that for any $\delta\in(0,\frac{1}{2})$ there exists a Tadmor--Ritt operator $T$ which satisfies square function estimates and
\begin{equation*}
 C(T,m,n)\gtrsim \left(\log\frac{n+1}{m+1}\right)^{\delta}
 \end{equation*}
 holds, where $C(T,m,n)$ is defined in \eqref{eq:C(T,m,n)}. Note that the involved constants depend on $\delta$.
\end{theorem}
An open question is whether there exists an $R$-Ritt operator on a Banach space such that $T$ satisfies square function estimates and $C(T,m,n)\gtrsim \left(\log\frac{n+1}{m+1}\right)^{\frac{1}{2}}$.\medskip\\ 
By $c_{1,T}\lesssim C(T)^{3}$, see \cite{Vitse2005b}, and $c\lesssim Pb(T)^{3}c_{1,T}$, see \cite[Proof of Theorem 7.3]{LeMerdy2014a}, we can track $C(T)$ in the constants of the estimate in Theorem \ref{thm3}. This yields a $C(T)$-dependence, which seems far from being sharp. Hence, the answer to \ref{itsharpness4} is probably `no'.
%\subsection{From discrete-time to continuous-time}
\section{Further results}\label{sec:Furtherresults} \label{subsec:SharpnessEstimates}
As a direct corollary of the improvements of Vitse's result, we get the following result for the Besov space functional calculus of $T$, which in turn is a slight improvement of \cite[Theorem 2.2]{Vitse2005b}.
For details of the following notions and facts see \cite{Vitse2005b} and the references therein.
Recall that the \textit{Besov space} $B_{\infty,1}(\D)$ is defined by the functions $f\in H(\D)$ such that 
\begin{equation*}
	\|f\|_{B} :=\|f\|_{\infty,\D}+\int_{0}^{1}\max_{\alpha}|f'(re^{i\alpha})|dr<\infty.
\end{equation*}
It is well known that there exists an equivalent definition via the dyadic decomposition $f=\sum_{n=0}^{\infty}W_{n}\ast f$, where $W_{n}$, $n\geq1$ are shifted Fejer type polynomials, whose Fourier coefficients $\widehat{W}_{n}(k)$ are the integer values of the triangular-shaped function supported in $[2^{n-1},2^{n+1}]$ with peak $\widehat{W}_{n}(2^{n})=1$ and $W_{0}(z)=1+z$. Here $(g\ast f)(z)=\sum_{k=0}^{\infty}\hat{g}(k)\hat{f}(k)z^{k}$. Then,
\begin{equation*}
	f\in B_{\infty,1}(\D) \iff f\in H(\D) \text{ and } \|f\|_{\ast}=\sum_{n=0}^{\infty}\|W_{n}\ast f\|_{\infty,\D}<\infty.
\end{equation*}
Since $W_{n}\ast f$ is a polynomial, we can use the $\|\cdot\|_{\infty,\D}$-estimate of Theorem \ref{thm2} to derive $B_{\infty,1}(\D)$-functional calculus estimates. This follows the same lines as in \cite{Vitse2005b}, however, using the improved constant dependence of our result in Theorem \ref{thm2}.
\begin{theorem}\label{thm4}
Let $T$ be a Tadmor--Ritt operator on a Banach space $X$. Then,
\begin{equation}\label{eq:thm4}
	\|f(T)\| \lesssim C(T)(\log(C(T)+1)) \|f\|_{\ast} 
\end{equation}i.e.,
for all $f\in B_{\infty,1}(\D)$, where $f(T)$ is defined by $\sum_{n=0}^{\infty}(W_{n}\ast f)(T)$.
\end{theorem}
\begin{proof}
Since $W_{n}\ast f \in H^{\infty}[2^{n-1},2^{n+1}]$ for $n\geq1$ and $W_{0}\ast f\in H^{\infty}[0,1]$, see Remark \ref{rem1} for the definition of $H^{\infty}[m,n]$, we can apply Theorem \ref{thm2} to derive
\begin{equation*}
	\|(W_{n}\ast f)(T)\| \leq a C(T)\left(2\log C(T)+b+\log\frac{2^{n+1}+1}{2^{n-1}+1}\right)\|W_{n}\ast f\|_{\infty,\D},
\end{equation*}
for $n\geq1$, with absolute constants $a,b>0$. Clearly, $\frac{2^{n+1}+1}{2^{n-1}+1}\leq 5$. Analogously, $\|(W_{0}\ast f)(T)\|$ can be estimated. Thus, $\sum_{n=0}^{\infty}\|(W_{n}\ast f)(T)\| \lesssim \|f\|_{\ast}$, and hence, $f(T)$ is well-defined with 
\begin{equation*}
	\|f(T)\| \leq a C(T)\left(2\log C(T)+b+\log5\right)\| f\|_{\ast}.
\end{equation*}
\end{proof}
In \cite[Theorem 2.5]{Vitse2005b} a similar $\|\cdot\|_{B_{\infty,1}(\D)}$-estimate as in \eqref{eq:thm4} is derived, but with a $C(T)$-dependence of $C(T)^{5}$.

%Section More general operators: Kreiss operators on Banach spaces. With modern techniques, so get that the are powerbounded
%Worst case? O(n)?! Even on Hilbert spaces?! (Foguel? Halmos?) gamma boundedness of operators.
%Trefethen: stability on general domains

%\begin{figure}
%% Use the relevant command to insert your figure file.
%% For example, with the graphicx package use
%  \includegraphics{example.eps}
%\end{figure}
%%
%% For two-column wide figures use
%\begin{figure*}
%% Use the relevant command to insert your figure file.
%% For example, with the graphicx package use
%  \includegraphics[width=0.75\textwidth]{example.eps}
%\end{figure*}
%\begin{table}
%\begin{tabular}{lll}
%\hline\noalign{\smallskip}
%\end{tabular}
%\end{table}

\subsection*{Acknowledgements}
The author would like to express his deepest gratitude to Hans Zwart for numerous discussions, and many very helpful comments on the manuscript. \newline He is very thankful to Eitan Tadmor for an interesting discussion in November 2014, and for sharing with him the proof of a result in \cite{TadmorLAA}. He is grateful to Joseph Ball for being his host at the Department of Mathematics at Virginia Tech in fall 2014, where parts of this manuscript were written. He would also like to thank Mark Embree for inspiring discussions on the Kreiss Matrix Theorem during that time.


%\begin{thebibliography}{}
%%
%% and use \bibitem to create references. Consult the Instructions
%% for authors for reference list style.
%\end{thebibliography}

 %\bibliographystyle{plain} 
   %Style abbrv
  % \bibliography{refsBigProject}

\end{document}